\documentclass[12pt]{amsart}
\usepackage[all]{xy}
\usepackage{a4wide}
\usepackage{mathtools}
\usepackage{amssymb}
\usepackage{amsthm}
\usepackage{amsmath}
\usepackage{amscd}
\usepackage{verbatim}
\usepackage[all]{xy}

\usepackage{amsthm}
\newtheorem{thm}{Theorem}[section]
\newtheorem{lem}[thm]{Lemma}
\newtheorem{prop}[thm]{Proposition}
\newtheorem{cor}[thm]{Corollary}

\theoremstyle{definition}
\newtheorem{dfn}[thm]{Definition}

\theoremstyle{remark}
\newtheorem{rem}[thm]{Remark}

\numberwithin{equation}{section}

	\newcommand{\R}{\mathbb{R}}			
	\newcommand{\C}{\mathbb{C}}			
	\newcommand{\Z}{\mathbb{Z}}				
	\newcommand{\Q}{\mathbb{Q}}       
	\renewcommand{\H}{\mathbb{H}}

	\newcommand{\abs}[1]{\left|#1\right|} 		

	\DeclareMathOperator{\tr}{tr}

    \renewcommand{\Re}{\text{Re}\,}
    \renewcommand{\Im}{\text{Im}\,}

	\DeclareMathOperator{\SL}{SL}
	
	\DeclareMathOperator{\sgn}{sgn}
	\DeclareMathOperator{\id}{id}

	\DeclareMathOperator{\e}{\mathfrak{e}}
	\DeclareMathOperator{\reg}{reg}
	\DeclareMathOperator{\erfc}{erfc}
	\DeclareMathOperator{\Mp}{Mp}
	
	\DeclareMathOperator{\CT}{CT}

	\DeclareMathOperator{\calF}{\mathcal{F}}
	\DeclareMathOperator{\calQ}{\mathcal{Q}}
	\DeclareMathOperator{\Iso}{Iso}
	\DeclareMathOperator{\spann}{span}
	\newcommand{\IM}{I^{\mathrm{M}}}
\newcommand{\ISh}{I^{\mathrm{Sh}}}

\DeclareMathOperator{\E}{\mathcal{E}}
\DeclareMathOperator{\Ei}{Ei}

\begin{document}

\title[]{Shintani theta lifts of harmonic Maass forms}

\author{Claudia Alfes-Neumann and Markus Schwagenscheidt}
\maketitle

\begin{abstract}
	We define a regularized Shintani theta lift which maps weight $2k+2$ ($k \in \Z, k \geq 0$) harmonic Maass forms for congruence subgroups to (sesqui-)harmonic Maass forms of weight $3/2+k$ for the Weil representation of an even lattice of signature $(1,2)$. We show that its Fourier coefficients are given by traces of CM values and regularized cycle integrals of the input harmonic Maass form. Further, the Shintani theta lift is related via the $\xi$-operator to the Millson theta lift studied in our earlier work. We use this connection to construct $\xi$-preimages of Zagier's weight $1/2$ generating series of singular moduli and of some of Ramanujan's mock theta functions. \end{abstract}

\setcounter{tocdepth}{1}
\tableofcontents

\section{Introduction}

\subsection{The Shintani theta lift} The classical Shimura-Shintani correspondence \cite{sh, shintani}, as refined by Kohnen \cite{kohnen80,kohnen82}, establishes a Hecke-equivariant isomorphism between the space $S_{2k+2}$ of cusp forms of integral weight $2k+2$ for $\Gamma=\SL_{2}(\Z)$ and the space $S_{3/2+k}$ of cusp forms of half-integral weight $3/2 + k$ for $\Gamma_{0}(4)$ satisfying the Kohnen plus space condition, for all integers $k \geq 0$. For higher level $N$ the correspondence gives isomorphisms between spaces of newforms of weight $2k+2$ for $\Gamma_{0}(N)$ and cuspidal Jacobi newforms of weight $k+2$ and index $N$, see \cite{gkz, skoruppazagier}. We restrict to level $N=1$ for simplicity in the introduction, but we will work with arbitrary congruence subgroups of $\SL_{2}(\Z)$ in the body of the paper by using the language of vector valued modular forms for the Weil representation of an even lattice of signature $(1,2)$. 

The isomorphism $S_{2k+2} \cong S_{3/2+k}$ can be realized by taking suitable linear combinations of Shintani theta lifts 
\begin{align*}
\ISh_{\Delta}:S_{2k+2} \to S_{3/2+k}, \qquad \ISh_{\Delta}(G,\tau) = \int_{\Gamma\backslash \H}G(z)\overline{\Theta_{Sh,\Delta}(\tau,z)}y^{2k+2} \frac{dx \, dy}{y^{2}},
\end{align*}
where $\tau = u+iv, z = x + iy \in \H$, and $\Theta_{Sh,\Delta}(\tau,z)$ is the non-holomorphic Shintani theta function, twisted by a fundamental discriminant $\Delta$ with $(-1)^{k+1}\Delta > 0$. It is explicitly given by
\begin{align*}
\Theta_{Sh,\Delta}(\tau,z) = 2v^{1/2} \! \! \! \! \! \!\sum_{\substack{D \in \Z \\ \sgn(\Delta)D \equiv 0,1 (4)}}\sum_{Q \in \calQ_{|\Delta|D}}\chi_{\Delta}(Q)\frac{Q(\bar{z},1)^{k+1}}{|\Delta|^{(k+1)/2}y^{2k+2}}e^{-4\pi v \frac{|Q(z,1)|^{2}}{|\Delta|y^{2}}}e^{-2\pi i D \tau},
\end{align*}
where $\calQ_{|\Delta|D}$ denotes the set of all integral binary quadratic forms $Q(x,y) = ax^{2} + bxy + cy^{2}$ (with $a,b,c \in \Z$) of discriminant $|\Delta|D = b^{2}-4ac$, and $\chi_{\Delta}$ is the generalized genus character from \cite{gkz} which is defined for $Q \in \calQ_{|\Delta|D}$ by
\begin{align*}
\chi_{\Delta}(Q) = \begin{dcases}
\left(\frac{\Delta}{n}\right), & \text{if $(a,b,c,\Delta) = 1$ and $Q$ represents $n$ with $(n,\Delta) = 1$,} \\
0, & \text{otherwise.}
\end{dcases}
\end{align*} 
The Shintani theta function transforms like a modular form of weight $2k+2$ for $\Gamma$ in $z$ and, by Poisson summation, its complex conjugate transforms like a modular form of weight $3/2+k$ for $\Gamma_{0}(4)$ in $\tau$. Further, as a function of $z$ it is of moderate growth at the cusp $\infty$. Thus the theta integral is well-defined and converges absolutely for $G \in S_{2k+2}$. 

The Shintani theta lift of cusp forms has first been studied by Shintani \cite{shintani}, Niwa \cite{niwa} and Cipra \cite{cipra} to refine the original correspondence of Shimura \cite{sh}.
Its Fourier expansion can be computed from the theta integral and is given by the generating series of twisted traces of geodesic cycle integrals of $G \in S_{2k+2}$,
\begin{align*}
\ISh_{\Delta}(G,\tau) = \frac{(-1)^{k}}{|\Delta|^{(k+1)/2}}\! \! \! \! \! \!\sum_{\substack{D > 0 \\ \sgn(\Delta)D \equiv 0,1(4)}} \left(\sum_{Q \in \calQ_{|\Delta|D}/\Gamma}\chi_{\Delta}(Q)\int_{\Gamma_{Q} \backslash c_{Q}}G(z)Q(z,1)^{k}dz \right) e^{2\pi i D \tau},
\end{align*}
where $c_{Q}$ denotes the geodesic in $\H$ defined by $a|z|^{2} + bx + c = 0$, which is a semi-circle oriented from $(-b-\sqrt{|\Delta|D})/2a$ to $(-b+\sqrt{|\Delta|D})/2a$ if $a \neq 0$, and the vertical line from $-\frac{c}{b}$ to $i\infty$ if $b > 0$ and from $i\infty$ to $-\frac{c}{b}$ if $b < 0$, if $a = 0$. On the one hand, this shows that $\ISh_{\Delta}(G,\tau)$ really is a cusp form, and on the other hand, it gives a conceptual proof for the modularity of an interesting generating series, which is in fact one of the major applications of similarly defined theta lifts. 

Recently, the Shintani theta lift has been generalized to weakly holomorphic modular forms by Bringmann, Guerzhoy and Kane \cite{briguka2,briguka}, again yielding cusp forms, and to differentials of the third kind by Bruinier, Funke, Imamoglu and Li \cite{bifl}, giving real analytic modular forms of weight $3/2$ whose holomorphic parts are generating series of traces of winding numbers of geodesics with respect to the poles of the input differential. The aim of the present work is to generalize the Shintani theta lift to harmonic Maass forms. Recall from \cite{brfu04} that a smooth function $G : \H \to \C$ is called a harmonic Maass form of weight $\kappa \in \Z$ for $\Gamma$ if it is annihilated by the invariant Laplace operator 
 \begin{align}\label{eq Laplace operator}
 \Delta_{\kappa}=-y^2\left(\frac{\partial^2}{\partial x^2}+\frac{\partial^2}{\partial y^2}\right)
+i\kappa y\left(\frac{\partial}{\partial x}+i\frac{\partial}{\partial y}\right),
\end{align}
transforms like a modular form of weight $\kappa$ under $\Gamma$, and is at most of linear exponential growth at the cusp $\infty$. A harmonic Maass form $G$ of weight $\kappa$ splits into a holomorphic and a non-holomorphic part $G = G^{+}+G^{-}$ having Fourier expansions of the form
 \begin{align}\label{eq Fourier expansion harmonic Maass form}
 G^{+}(z) = \sum_{\substack{n \in \Z \\ n \gg -\infty}}a_{G}^{+}(n)e^{2\pi i n z}, \quad G^{-}(z) = a_{G}^{-}(0)y^{1-\kappa}+\sum_{\substack{n \in \Z \setminus \{0\} \\ n \ll \infty}}a_{G}^{-}(n)\E_{\kappa}(4\pi n y)e^{2\pi i nz},
 \end{align}
 with coefficients $a_{G}^{\pm}(n) \in \C$, and $y^{1-\kappa}$ replaced by $\log(y)$ if $\kappa = 1$. Here the function $\E_{\kappa}(y)$ could be any antiderivative of $e^{y}(-y)^{-\kappa}$, and for our purposes it is convenient to set
\begin{align*}
		\E_{\kappa}(y) = \begin{dcases}
		\Gamma(1-\kappa,-y), & \text{if } \kappa \leq 0, \\
		\frac{(-1)^{\kappa+1}}{(\kappa-1)!}\left(e^{y}\sum_{j=0}^{\kappa-2}y^{-j-1}j!-\Ei(y)\right),& \text{if } \kappa > 0,
		\end{dcases}
\end{align*}
where $\Gamma(s,y) = \int_{y}^{\infty}e^{-t}t^{s-1}dt$ is the incomplete Gamma function and $\Ei(y) = -\int_{-y}^{\infty}e^{-t}\frac{dt}{t}$ is the exponential integral, which is defined using the Cauchy principal value for $y > 0$. The space of harmonic Maass forms of weight $\kappa$ is denoted by $H_{\kappa}$. Harmonic Maass forms of half-integral weight for $\Gamma_{0}(4)$ satisfying the Kohnen plus space condition are defined analogously. We also require 
  the antilinear differential operator
 \[
 \xi_{\kappa}G = 2iy^{\kappa}\overline{\frac{\partial}{\partial \bar{z}}G(z)},
 \]
 which defines a surjective map $H_{\kappa} \to M_{2-\kappa}^{!}$ and acts on the Fourier expansion of $G \in H_{\kappa}$ by
\[
\xi_{\kappa}G = \xi_{\kappa}G^{-} = (1-\kappa)\overline{a_{G}^{-}(0)}-\sum_{\substack{n \gg 0 \\ n \neq 0}}(4\pi n)^{1-\kappa}\overline{a_{G}^{-}(-n)}e^{2\pi i n z}.
\]
It is related to the Laplace operator by $\Delta_{\kappa} = -\xi_{2-\kappa}\xi_{\kappa}$. We let $H_{\kappa}^{+}$ be the subspace consisting of those harmonic Maass forms in $H_{\kappa}$ that map to a cusp form in $S_{2-\kappa}$ under $\xi_{\kappa}$.

Since $\Theta_{Sh,\Delta}(\tau,z)$ is of moderate growth at $\infty$, the theta integral does usually not converge for $G \in H_{2k+2}$, and has to be regularized in a suitable way to give it a meaning. Following Borcherds \cite{borcherds}, we define the regularized Shintani theta lift of $G \in H_{2k+2}$ by
 \begin{align*}
 \ISh_{\Delta}(G,\tau) = \CT_{s = 0}\left[\lim_{T \to \infty}\int_{\calF_{T}}G(z)\overline{\Theta_{Sh,\Delta}(\tau,z)}y^{2k+2-s}\frac{dx\, dy}{y^{2}}\right],
 \end{align*}
 where $\calF_{T} = \{z \in \H: |x| \leq 1/2, |z| \geq 1, y \leq T\}$ is a truncated fundamental domain for $\Gamma \backslash \H$ and $\CT_{s = 0}F(s)$ denotes the constant term in the Laurent expansion at $s = 0$ of a meromorphic function $F(s)$. Our first main result is then the following.
 
 \begin{thm}
 	The regularized Shintani theta lift $\ISh_{\Delta}(G,\tau)$ of a harmonic Maass form $G \in H_{2k+2}$ exists and defines a harmonic Maass form in $H_{3/2+k}$. If $G \in M_{2k+2}^{!}$ is a weakly holomorphic modular form then $\ISh_{\Delta}(G,\tau) \in M_{3/2+k}$ is a holomorphic modular form, and if in addition $a_{G}^{+}(0) = 0$ then $\ISh_{\Delta}(G,\tau) \in S_{3/2+k}$ is a cusp form.
 \end{thm} 
 
 \begin{rem}
 	For $k = 0$ and higher level the Shintani theta lift of a harmonic Maass form of weight $2$ is in general not harmonic, but maps to a linear combination of unary theta series of weight $3/2$ under $\Delta_{3/2}$, i.e., it is a sesquiharmonic Maass form. Note that there are no unary theta series of weight $3/2$ for $\Gamma_{0}(4)$, and thus the Shintani theta lift is harmonic in this case.
\end{rem}

The above result follows from the basic properties of the Shintani theta lift established in Section~\ref{section theta lift} together with its Fourier expansion given in Theorem \ref{theorem fourier expansion shintani}.
 
 In our earlier work \cite{alfesschwagen} we investigated the so-called Millson theta lift $\IM_{\Delta}(F,\tau)$ from weight $-2k$ to weight $1/2-k$ harmonic Maass forms, with $F \in H_{-2k}^{+}$ assumed to map to a cusp form under $\xi_{-2k}$. We showed that it is related to the Shintani theta lift via the differential equation
 \begin{align*}
  \xi_{1/2-k,\tau}\IM_{\Delta}(F,\tau) = -4^{k}\sqrt{|\Delta|}\ISh_{\Delta}(\xi_{-2k,z}F,\tau).
 \end{align*}
 Note that both sides of the above equation are cusp forms. Since $\xi: H_{-2k}^{+} \to S_{2k+2}$ is surjective, this relation can be used to construct $\xi$-preimages of cusp forms of weight $3/2+k$ arising as Shintani theta lifts of cusp forms of weight $2k+2$. 
  Conversely, we show the following relation which implies that the Shintani theta lift can be used to construct $\xi$-preimages of weakly holomorphic modular forms of weight $1/2-k$ arising as Millson theta lifts of weakly holomorphic modular forms of weight $-2k$.
 
 \begin{prop}\label{prop:xidiagramintroduction}
 	The Millson and the Shintani theta lift satisfy the differential equation
	\[
	\xi_{3/2+k,\tau}\ISh_{\Delta}(G,\tau) = -\frac{1}{4^{k+1}\sqrt{|\Delta|}}\IM_{\Delta}(\xi_{2k+2,z}G,\tau)
	\]
	for every harmonic Maass form $G \in H_{2k+2}$.
 \end{prop}
 
 The proof follows from an analagous differential equation satisfied by the Millson and the Shintani theta functions and an application of Stokes' theorem. See Proposition \ref{proposition xi diagram} for the general result. We will illustrate this result in Section \ref{section application fD} below. 
 
 \subsection{Traces of CM values and regularized cycle integrals of harmonic Maass forms}
 The second main result of this work is the computation of the Fourier expansion of the Shintani theta lift. It turns out that it is the simultaneous generating series of traces of (suitably regularized) cycle integrals and traces of CM values of (iterated derivatives of) $G \in H_{2k+2}$, whose definition we now explain. 
 
 For $Q \in Q_{|\Delta|D}$ with $D < 0$ we let $z_{Q} \in \H$ be the CM point associated to $Q$ which is characterized by $Q(z_{Q},1) = 0$. In this case the stabilizer of $Q$ in $\overline{\Gamma} = \Gamma/\{\pm 1\}$ is finite. For $D < 0$ such that $\sgn(\Delta)D$ is a discriminant and a $\Gamma$-invariant function $F$ we define a twisted trace function by
 \[
 \tr_{\Delta}^{+}(F,D) = \sum_{Q \in \calQ_{|\Delta|D}^{+}/\Gamma}\chi_{\Delta}(Q)\frac{F(z_{Q})}{|\overline{\Gamma}_{Q}|},
 \]
 where $\calQ_{|\Delta|D}^{+}$ consists of all positive definite ($a > 0$) forms in $\calQ_{|\Delta|D}$. If $F = 1$ and $\Delta = 1$ then $\tr_{1}^{+}(1,D) = H(D)$ is a Hurwitz class number.

 For $D > 0$ we let $c_{Q}$ be the geodesic in $\H$ defined by $a|z|^{2} + bx + c = 0$, with the orientation as explained above. If $|\Delta|D$ is a square then the stabilizer of $Q$ in $\overline{\Gamma}$ is trivial and $c_{Q}$ is an infinite geodesic in the modular curve $\Gamma \backslash \H$, and if $|\Delta|D$ is not a square then the stabilizer of $Q$ in $\overline{\Gamma}$ is infinite cyclic and $\Gamma_{Q} \backslash c_{Q}$ is a closed geodesic in $\Gamma \backslash \H$. For $D > 0$ such that $\sgn(\Delta)D$ is a discriminant and $|\Delta|D$ is not a square we define a twisted trace function of a harmonic Maass form $G \in H_{2k+2}$ by
 \[
 \tr_{\Delta}(G,D) = \sum_{Q \in \calQ_{|\Delta|D}/\Gamma}\chi_{\Delta}(Q)\int_{\Gamma_{Q}\backslash c_{Q}}G(z)Q(z,1)^{k}dz.
 \]
 If $|\Delta|D$ is a square then the cycle integrals appearing above typically diverge since $c_{Q}$ is infinite in $\Gamma \backslash \H$ and $G$ is of exponential growth at the cusp $\infty$. To explain the appropriate regularization we assume for simplicity that $Q(x,y) = xy$, so that $c_{Q}$ is the positive imaginary axis. For the general definition we refer the reader to Section~\ref{section regularized cycle integrals}. If we suppose for the moment that $G \in S_{2k+2}$ is a cusp form, then we can write
 \[
 \int_{c_{Q}}G(z)Q(z,1)^{k}dz = i^{k+1}\int_{0}^{\infty}G(iy)y^{k}dy = i^{k+1}(1+(-1)^{k+1})\lim_{T \to \infty} \int_{1}^{T}G(iy)y^{k}dy.
 \]
 For fixed $T > 0$ the last integral also exists for a harmonic Maass form $G \in H_{2k+2}$. Further, by plugging in the Fourier expansion of $G \in H_{2k+2}$, the integral can explicitly be evaluated (see Lemma \ref{lemma regularized cycle integral}). By subtracting all the parts which are growing as $T \to \infty$, we obtain the following definition.
 
 \begin{dfn}
 	Suppose that $Q(x,y) = xy$, and let $G \in H_{2k+2}$. Pick a positive real number $T > 0$. Then for $k > 0$ the regularized cycle integral of $G$ along $c_{Q} = i\R_{> 0}$ is defined by
 	\begin{align*}
		&\int_{c_{Q}}^{\reg}G(z)Q(z,1)^{k}dz = i^{k+1}(1+(-1)^{k+1})\\
		&\quad \times \bigg(\int_{1}^{T}G(iy)y^{k}dy - a_{G}^{+}(0)\frac{T^{k+1}}{k+1}  + \sum_{n \neq 0}a_{G}^{+}(n)\frac{\Gamma(k+1,2\pi n T)}{(2\pi n )^{k+1}}+ a_{G}^{-}(0)\frac{T^{-k}}{k} \\
		&\quad \quad + \sum_{n \neq 0}a_{G}^{-}(n)\bigg(\E_{2k+2}(4\pi nT)\frac{\Gamma(k+1,2\pi n T)}{(2\pi n)^{k+1}}  -\frac{2^{-2k-1}k!}{(2\pi n)^{k+1}}\sum_{j=0}^{k} \frac{(-1)^{j}}{j!}\E_{2k+2-j}(2\pi nT)\bigg)\bigg).
	\end{align*}
	For $k = 0$ it is defined by the same formula\footnote{For even $k$ the above regularized cycle integral actually vanishes for trivial reasons. Nevertheless, this is not true for arbitrary quadratic forms of square discriminant or harmonic Maass forms of higher level, and the above formula gives the correct idea.}, but $\frac{T^{-k}}{k}$ has to be replaced by $-\log(T)$. The regularized cycle integral is independent of the choice of $T > 0$.
 \end{dfn}
 
 \begin{rem}
 	\begin{enumerate}
		\item By plugging in $T = 1$ the formula further simplifies. If $G \in M_{2k+2}^{!}$ is weakly holomorphic and satisfies $a_{G}^{+}(0) = 0$, this agrees with the definition of the regularized cycle integral of $G$ given in \cite{brifrike}.
		\item The same idea has been used in \cite{bif} to define regularized cycle integrals of harmonic Maass forms of weight $0$. As in Theorem 3.2 of \cite{bif} the regularized cycle integrals can be rewritten in terms of line integrals of $G$ and $\xi_{2k+2}G$ against polygamma functions and Bernoulli polynomials, see Proposition \ref{prop regularized cycle integrals alternative}. 
		\item It is not yet clear that this naive approach gives the correct regularization. It will be justified by the fact that the regularized cycle integrals appear in the Fourier expansion of the Shintani lift. 
	\end{enumerate}
 \end{rem}
 
 We define $\tr_{\Delta}(G,D)$ for $|\Delta|D > 0$ being a square in an analogous way as before, but using the regularized cycle integrals of $G$. Similarly as in \cite{brifrike} and \cite{bif}, we may view this trace as a replacement for the central critical value of the (non-existent) $L$-series of a harmonic Maass form of weight $2k+2$. Correspondingly, we define the $\Delta$-th twisted central critical $L$-value of $G \in H_{2k+2}$ by
 \begin{align}\label{eq LValue}
 L_{\Delta}^{*}(G,k+1) = \sqrt{|\Delta|}\tr_{\Delta}(G,-\Delta),
 \end{align}
 which is inspired by the analogous formula for the central value $s = k+1$ of the completed twisted $L$-series 
 \[
 L_{\Delta}^{*}(F,s) = (2\pi)^{-s}\Gamma(s)|\Delta|^{s}\sum_{n > 0}\left(\frac{\Delta}{n}\right)a_{F}(n)n^{-s}
 \]
 of a cusp form $F \in S_{2k+2}$, see the proof of Corollary~1 in \cite{kohnen85}. We give an explicit example in Section~\ref{section application fD} below.
 
 \subsection{The Fourier expansion of the Shintani lift}
 
	We can now state the Fourier expansion of the Shintani lift, which is the main result of this work. To simplify the notation we only give a formula for its holomorphic part in the introduction.

 \begin{thm}
 	Let $k \in \Z$ with $k \geq 0$ and let $G \in H_{2k+2}$. The Fourier expansion of the holomorphic part of the Shintani theta lift $\ISh_{\Delta}(G,\tau)$ is given by
 	\begin{align*}
	|\Delta|^{(k+1)/2}\ISh_{\Delta}(G,\tau)^{+} =&(-1)^{k+1}|\Delta|^{-k/2}a_{G}^{+}(0)L_{\Delta}(-k) + (-1)^{k}\sum_{D > 0}\tr_{\Delta}(G,D)e^{2\pi i D\tau},
	\end{align*}
	where $L_{\Delta}(s) = \sum_{n=1}^{\infty}\left(\frac{\Delta}{n}\right)n^{-s}$ for $\Re(s) > 1$ is a Dirichlet $L$-series, and $\tr_{\Delta}(G,D)$ are twisted traces of regularized cycle integrals of $G$.
 \end{thm}

 \begin{rem}
 \begin{enumerate}
 \item The non-holomorphic part is determined by the relation to the Millson lift from Proposition~\ref{prop:xidiagramintroduction} and the fact that the Fourier expansion of the Millson lift of the weakly holomorphic modular form $\xi_{2k+2}G$ is given by the generating series of traces $\tr_{\Delta}^{+}(R_{-2k}^{k}\xi_{2k+2},D)$, where $R_{-2k}^{k} = R_{-2}\circ R_{-4} \circ\dots\circ R_{-2k}$ (with $R_{0}^{0} = \id$) is an iterated version of the Maass raising operator $R_{\kappa} = 2i\frac{\partial}{\partial z} + \kappa y^{-1}$. We give the full Fourier expansion in Theorem~\ref{theorem fourier expansion shintani}.
 \item In \cite{briguka2,briguka} the Shintani lift was generalized to weakly holomorphic modular forms with vanishing constant terms, and it was shown to yield cusp forms whose Fourier expansions are given by regularized cycle integrals of the input. However, the authors did not use a theta lift approach, but they showed that the regularized cycle integrals of a weakly holomorphic modular form $F \in M_{2k+2}^{!}$ with $a_{F}(0) = 0$ essentially agree with the cycle integrals of the cusp form $\xi_{-2k}\mathcal{F} \in S_{2k+2}$, where $\mathcal{F} \in H_{-2k}^{+}$ is any harmonic Maass form with $D^{2k+2}\mathcal{F} = F$, with $D = \frac{1}{2\pi i}\frac{\partial}{\partial z}$. 
 This approach does not work for harmonic Maass forms.
 \end{enumerate}
 \end{rem}
 
 The crucial ingredient for the computation of the above Fourier expansion is a preimage under the Laplace operator $\Delta_{2k+2}$ of the Shintani Schwartz function
 \[
 \varphi_{Sh}^{0}(Q,\tau,z) = 2v^{1/2}\frac{Q(\bar{z},1)^{k+1}}{|\Delta|^{(k+1)/2}y^{2k+2}}e^{-4\pi v\left(\frac{|Q(z,1)|^{2}}{|\Delta|y^{2}}-D\right)}.
 \]
 Note that the Shintani theta function is given by 
 \[
 \Theta_{Sh,\Delta}(\tau,z) = \sum_{\substack{D \in \Z \\ \sgn(\Delta)D \equiv 0,1(4)}}\sum_{Q \in \calQ_{|\Delta|D}}\chi_{D}(Q)\varphi_{Sh}^{0}(Q,\tau,z)e^{-2\pi i D\bar{\tau}}.
 \]
 We define the quantity $p_{z}(Q) = -y^{-1}(a|z|^{2}+bx + c)$, which vanishes exactly on the geodesic $c_{Q}$. 
 \begin{prop}
 	For $Q \in \calQ_{|\Delta|D}$ and $z \in \H$ with $Q(z,1) \neq 0$ the function
	\begin{align*}
	\eta(Q,\tau,z) = -\frac{1}{2Q(z,1)^{k+1}}\int_{\frac{|p_{z}(Q)|}{\sqrt{|\Delta|}}}^{\infty}\left(t^{2}+D\right)^{k}\erfc\left(\sqrt{\pi v}t\right)dt,
	\end{align*}
	with $\erfc(x) = \frac{2}{\sqrt{\pi}}\int_{x}^{\infty}e^{w^{2}}dw$, satisfies
	\begin{align*}
	\xi_{2k+2}\eta(Q,\tau,z) = \frac{Q(z,1)^{k}}{2|\Delta|^{k+1/2}} \sgn(p_{z}(Q))\erfc\left(\sqrt{\pi v}\frac{|p_{z}(Q)|}{\sqrt{|\Delta|}}\right)
	\end{align*}
	and
	\begin{align*}
	\Delta_{2k+2}\eta(Q,\tau,z) = \varphi_{Sh}^{0}(Q,\tau,z).
	\end{align*}
 \end{prop}
 
 The proposition can be proved by a direct computation. The most important point for us is that $\eta(Q,\tau,z)$ has a singularity at the CM point $z_{Q}$ if $D < 0$, and $\xi_{2k+2}\eta(Q,\tau,z)$ has a jump discontinuity along the geodesic $c_{Q}$ if $D > 0$.
 
 To explain the main idea of the computation of the Fourier expansion of $\ISh_{\Delta}(G,\tau)$ we ignore the necessary regularization and all convergence issues for the moment. Then, roughly speaking, the $D$-th Fourier coefficient for $D \neq 0$ of the Shintani theta lift can be computed by the unfolding argument as
 \begin{align*}
 \sum_{Q \in \mathcal{Q}_{|\Delta|D}/\Gamma}\chi_{\Delta}(Q)\int_{\Gamma_{Q} \setminus \H}G(z)\overline{\varphi_{Sh}^{0}(Q,\tau,z)}y^{2k+2}\frac{dx \, dy}{y^{2}}.
 \end{align*}
 If we now plug in $\Delta_{2k+2}\eta(Q,\tau,z)$ for $\varphi_{Sh}^{0}(Q,\tau,z)$, we formally obtain, using Stokes' theorem and the fact that $G$ is harmonic, that
 \begin{align*}
 \int_{\Gamma_{Q} \setminus \H}G(z)\overline{\varphi_{Sh}^{0}(Q,\tau,z)}y^{2k+2}\frac{dx \, dy}{y^{2}} &= \int_{\partial (\Gamma_{Q}\backslash \H)}G(z)\xi_{2k+2}\eta(Q,\tau,z)dz \\
 &\quad -\overline{\int_{\partial (\Gamma_{Q}\backslash \H)}\xi_{2k+2}G(z)\eta(Q,\tau,z)dz}.
 \end{align*}
 Thus we can reduce the computation of the two-dimensional integral on the left-hand side to the computation of two one-dimensional boundary integrals. In fact, we have to be a bit more careful since $\eta(Q,\tau,z)$ and $\xi_{2k+2}\eta(Q,\tau,z)$ have singularities at CM points and geodesics, so before using Stokes' theorem we have to cut out small $\varepsilon$-neighbourhoods around these singularities, and afterwards let $\varepsilon$ go to $0$. This gives further boundary integrals, which yield the CM values of $R_{-2k}^{k}\xi_{2k+2}G$ and parts of the regularized cycle integrals of $G$. Due to the necessary regularization, the actual computation of the several boundary integrals is still very involved, and occupies a considerable part of this work. 
 
 The existence of the $\Delta_{2k+2}$-preimage $\eta(Q,\tau,z)$ of $\varphi_{Sh}^{0}(Q,\tau,z)$ is not at all clear, and was quite surprising to us. Without it, the computation of the Fourier expansion would have been even harder, and maybe unfeasible. Further, it opens a way to further generalizations of the Shintani theta lift. For example, it should be possible, using similar methods as presented here, to extend the Shintani theta lift to meromorphic modular forms or polar harmonic Maass forms. We hope to come back to this problem in the near future.

 \subsection{Application: $\xi_{3/2}$-preimages of weight $1/2$ weakly holomorphic modular forms}\label{section application fD}
 
 Let $k = 0$ and let $\Delta < 0$ be a fundamental discriminant. We let $J = j-744 \in M_{0}^{!}$ be the $j$-invariant without its constant term, and we let
 \[
 f_{|\Delta|}(\tau) = e^{2\pi i \Delta \tau} + \sum_{D < 0}\frac{1}{\sqrt{|D|}}\tr_{\Delta}^{+}(J,D)e^{-2\pi i D\tau}
 \]
 be the generating series of twisted singular moduli discovered by Zagier \cite{zagiertraces}. It is a weakly holomorphic modular form of weight $1/2$ for $\Gamma_{0}(4)$ satisfying the Kohnen plus space condition, and it can be constructed as the Millson theta lift $\IM_{\Delta}(J,\tau) = 2f_{|\Delta|}(\tau)$ of $J$, compare \cite{alfesschwagen}, Theorem 1.1. Let $\tilde{J} \in H_{2}$ be a $\xi_{2}$-preimage of $J$. Such a preimage exists by the surjectivity of the $\xi$-operator, but can also be explicitly constructed as a Maass-Poincar\'e series, see \cite{brhabil}, Section 1.3. The holomorphic part of its Shintani theta lift has the Fourier expansion
 \begin{align*}
 \sqrt{|\Delta|}\ISh_{\Delta}(\tilde{J},\tau)^{+}  &= -a_{\tilde{J}}^{+}(0)L_{\Delta}(0) + \sum_{D > 0}\tr_{\Delta}\big(\tilde{J},D\big)e^{2\pi i D\tau}.
 \end{align*}
 By Proposition \ref{prop:xidiagramintroduction} we have
 \[
 \xi_{3/2}\ISh_{\Delta}\big(\tilde{J},\tau\big) = -\frac{1}{4\sqrt{|\Delta|}}\IM_{\Delta}(J,\tau) = -\frac{1}{2\sqrt{|\Delta|}}f_{|\Delta|}(\tau),
 \]
 so the Shintani theta lift of $\tilde{J}$ yields a $\xi_{3/2}$-preimage of $f_{|\Delta|}$. This can be used to evaluate the regularized Petersson inner products $(f_{D},f_{|\Delta|})^{\reg}$ with $f_{D}$ for $-D < 0$ a discriminant. For the exact definition of this inner product, in particular for the case $D = |\Delta|$, we refer the reader to \cite{bringmanndiamantisehlen}. We will make use of the fact that the regularized inner product of two weakly holomorphic modular forms $f , g \in M_{\kappa}^{!}$ with $g = \xi_{2-\kappa}G$ for some $G \in H_{2-\kappa}$ can be evaluated as 
 \begin{align}\label{eq inner product evaluation}
 (f,g)^{\reg} = (f,\xi_{2-\kappa}G)^{\reg} = \sum_{D \in \Z}a_{f}(D)a_{G}^{+}(-D),
 \end{align}
 see \cite{bringmanndiamantisehlen}, Theorem 4.1. Since the Shintani theta lift does not have a holomorphic principal part, i.e. the coefficients of negative index of its holomorphic part vanish, we obtain the following result.
 
 \begin{prop}
 	For $\Delta < 0$ a fundamental discriminant and $-D < 0$ a discriminant we have the formula
	\[
	\big(f_{D},f_{|\Delta|}\big)^{\reg} = -2\tr_{\Delta}\big(\tilde{J},D\big).
	\]
 \end{prop}

	\begin{rem}
		\begin{enumerate}
		\item Interpreting $\tr_{\Delta}\big(\tilde{J},-\Delta\big)$ as a regularized $L$-value as in \eqref{eq LValue}, we obtain the striking identity
		\[
		\big(f_{|\Delta|},f_{|\Delta|}\big)^{\reg} = -\frac{2}{\sqrt{|\Delta|}}L_{\Delta}^{*}\big(\tilde{J},1\big).
		\]
		\item The regularized $L$-value of $\tilde{J}$ is explicitly given by
		\begin{align*}
		\frac{\pi}{|\Delta|}L_{\Delta}^{*}\big(\tilde{J},1\big) &= -\sum_{n\neq 0}\left(\frac{\Delta}{n}\right)\frac{a_{\tilde{J}}^{+}(n)}{n}e^{-2\pi n/|\Delta|} \\
		&\quad -\sum_{n\neq 0}\left(\frac{\Delta}{n}\right)\frac{a_{\tilde{J}}^{-}(n)}{n}\left(\E_{2}\left(\frac{4\pi n}{|\Delta|}\right)e^{-2\pi n/|\Delta|}-\frac{1}{2}\E_{2}\left(\frac{2\pi n}{|\Delta|}\right)\right).
		\end{align*}
		Note that $a_{\tilde{J}}^{-}(n) = 4\pi n \, a_{J}(-n)$. By adding a suitable weakly holomorphic modular form, we can assume that $a_{\tilde{J}}^{+}(n) = 0$ for $n < 0$. Then, by \eqref{eq inner product evaluation} we have $a_{\tilde{J}}^{+}(n) = (J_{n},J)^{\reg}$ for $n \geq 0$, where $J_{n} \in M_{0}^{!}$ is the unique modular function whose Fourier expansion starts $e^{-2\pi i nz} + O(e^{2\pi i z})$. 

		\item An analogous evaluation of the regularized Petersson norm of the weight $3/2$ generating series of singular moduli $g_{1}$ is given in \cite{bringmanndiamantisehlen}, Theorem 1.2.
		\item If $\Delta$ is not a fundamental discriminant then $f_{|\Delta|}$ can be written as a linear combination of twisted Millson theta lifts. This can be used to work out a formula for the above inner product for non-fundamental discriminants $\Delta < 0$ as well.
		\end{enumerate}
	\end{rem}

	We would like to point out that the above application was in fact one of main motivations and the starting point for the present work. Recently, Bruinier, Funke and Imamoglu \cite{bif} constructed a similar regularized theta lift (producing weight $1/2$ harmonic Maass forms) by integrating a weight $0$ harmonic Maass form against the Siegel theta function. Thereby they constructed $\xi_{1/2}$-preimages of the weight $3/2$ generating series of singular moduli $g_{\Delta}$ for fundamental discriminants $\Delta > 0$. The regularized Shintani theta lift provides $\xi_{3/2}$-preimages of the weight $1/2$ modular forms $f_{|\Delta|}$ for $\Delta < 0$.
 
 \subsection{Application: Shintani theta lifts of Eisenstein series}

 	Let $k = 0$ and let $\Delta < 0$ be a fundamental discriminant. 
	By the Dirichlet class number formula we have $L_{\Delta}(0) = \sqrt{|\Delta|}L_{\Delta}(1)/\pi = H(|\Delta|)$.

	Let
	\begin{align*}
	E_{2}^{*}(z) = 1-24\sum_{n \geq 1}\sigma_{1}(n)e^{2\pi i n z} - \frac{3}{\pi y}
	\end{align*}
	be the non-holomorphic Eisenstein series of weight $2$ for $\Gamma$. Note that $\xi_{2}E_{2}^{*} = \frac{3}{\pi}$. The Shintani theta lift of $E_{2}^{*}$ is a harmonic Maass form of weight $3/2$ having the Fourier expansion
	\begin{align*}
	\sqrt{|\Delta|}\ISh_{\Delta}(E_{2}^{*},\tau) &= -H(|\Delta|) + \sum_{D > 0}\tr_{\Delta}(E_{2}^{*},D)e^{2\pi i D \tau} \\ 
	&\quad +\frac{3}{2\pi}H(|\Delta|)v^{-1/2} + \sum_{D < 0}\frac{\tr_{\Delta}^{+}\left(\tfrac{3}{\pi},D\right)}{2\sqrt{|D|}}v^{-1/2}\beta_{3/2}(4\pi |D|v)e^{2\pi i D \tau},
	\end{align*}
	where $\beta_{3/2}(s) = \int_{1}^{\infty}e^{-st}t^{-3/2}dt$.
	We would like to compare this to Zagier's non-holomorphic Eisenstein series of weight $3/2$ defined by
	\begin{align*}
	E_{3/2}^{*}(\tau) = \sum_{D \geq 0}H(D)e^{2\pi i D\tau}+\frac{1}{16\pi}\sum_{n \in \Z}v^{-1/2}\beta_{3/2}(4\pi n^{2}v)e^{-2\pi i n^{2}\tau},
	\end{align*}
	where $H(0) = -\frac{1}{12}$ and $H(D) = 0$ if $-D \neq 0$ is not a discriminant, see \cite{zagiereisenstein}. Note that $\beta_{3/2}(0) = 2$. We see that the difference $\sqrt{|\Delta|}\ISh_{\Delta}(E_{2}^{*},\tau)-12H(|\Delta|)E_{3/2}^{*}(\tau)$ is a harmonic Maass form of weight $3/2$ whose constant coefficients of the holomorphic and non-holomorphic parts vanish. In particular, it maps to a cusp form of weight $1/2$ under $\xi_{3/2}$. Since $S_{1/2} = \{0\}$, we find that the difference is a holomorphic cusp of weight $3/2$, but $S_{3/2} = \{0\}$ as well, so we obtain that the difference vanishes. 
	
	\begin{prop} The Shintani theta lift of $E_{2}^{*}$ is given by
	\[
	\sqrt{|\Delta|}\ISh_{\Delta}(E_{2}^{*},\tau) = 12H(|\Delta|)E_{3/2}^{*}(\tau).
	\]
	\end{prop}
	
	By comparing the holomorphic parts of both sides, we obtain an interesting identity of cycle integrals of $E_{2}^{*}$ and Hurwitz class numbers.
	
	\begin{cor}
		For a fundamental discriminant $\Delta < 0$ and a discriminant $-D < 0$ we have
		\[
		\sum_{Q \in \calQ_{|\Delta| D}/\Gamma}\chi_{\Delta}(Q)\int_{\Gamma_{Q}\backslash c_{Q}}^{\reg}E_{2}^{*}(z)dz = 12H(|\Delta|)H(D).
		\]
	\end{cor}

	\begin{rem}
		\begin{enumerate}
			\item For $\Delta < 0$ and $-D < 0$ being coprime fundamental discriminants, the above formula goes back to Hecke. We refer to \cite{ditkronecker} for a nice overview over this and related formulas.
			\item The case $-D = \Delta$ is particularly interesting. The twisted trace of regularized cycle integrals of $E_{2}^{*}$ can explicitly be evaluated by choosing as representatives for $\calQ_{|\Delta|^{2}}/\Gamma$ the forms $Q(x,y) = |\Delta|xy + cy^{2}$ with $0 \leq c < |\Delta|$. For a fundamental discriminant $\Delta < 0$ we obtain the equality
		\[
		\frac{1}{12\sqrt{|\Delta|}}L_{\Delta}^{*}(E_{2}^{*},1) = \frac{2\sqrt{|\Delta|}}{\pi}\sum_{n > 0}\left( \frac{\Delta}{n}\right)\frac{\sigma_{1}(n)}{n}e^{-2\pi n/|\Delta|} = H(|\Delta|)^{2}.
		\]
		The second equality can also be derived from Dirichlet's class number formula.
			\item By comparing the non-holomophic parts of $E_{3/2}^{*}$ and the Shintani theta lift of $E_{2}^{*}$, we obtain for $-D > 0$ being a discriminant the formula
		\begin{align*}
		\frac{\tr_{\Delta}^{+}(1,D)}{\sqrt{|D|}} = \begin{cases}
		H(|\Delta|), & \text{if $|D|$ is a square}, \\
		0, & \text{otherwise,}
		\end{cases}
		\end{align*}
		which can also be verified directly.
		\end{enumerate}
	\end{rem}

	Similarly, we find for $k > 0$ that the Shintani theta lift of the normalized Eisenstein series $E_{2k+2}$ for $\Gamma$ is a multiple of the Cohen Eisenstein series of weight $3/2+k$ \cite{cohen}. This is well known, and also follows from the evaluation of cycle integrals of $E_{2k+2}$ along $c_{Q}$ in terms of the $L$-series $\zeta_{Q}(s)$ associated to $Q$ due to Kohnen and Zagier \cite{kohnenzagierrationalperiods} (see also \cite{funkemillson}, Section 9, and \cite{briguka}, Theorem 1.3 and Lemma 4.1). However, the lift of the non-holomorphic Eisenstein series $E_{2}^{*}$ seems to be new.

\subsection{Application: $\xi_{3/2}$-preimages of Ramanujan's mock theta functions}

	We would like to explain how the Shintani theta lift can be used to construct $\xi_{3/2}$-preimages of Ramanujan's mock theta functions. The idea is similar as in Section~\ref{section application fD}, so we will only give a sketch. Let
		\begin{align*}
		f(q) = 1 + \sum_{n=1}^{\infty}\frac{q^{n^{2}}}{\prod_{j=1}^{n}(1+q^{j})^{2}},\qquad w(q) = 1 + \sum_{n=1}^{\infty}\frac{q^{2n^{2}+2n}}{\prod_{j=1}^{n+1}(1-q^{2j-1})^{2}},
		\end{align*}
		be two of Ramanujan's mock theta functions of order $3$. Zwegers \cite{zwegerspaper} showed that they can be understood as components of the holomorphic part of a vector valued harmonic Maass form $H(\tau)$ of weight $1/2$ (see also \cite{brono}, Section~8, for an overview). In \cite{bruinierschwagenscheidt}, it was noticed that the Millson theta lift of the $\Gamma_{0}(6)$-invariant function
		\[
		F(z) = -\frac{1}{40}\cdot\frac{E_{4}(z)+4E_{4}(2z)-9E_{4}(3z)-36E_{4}(6z)}{(\eta(z)\eta(2z)\eta(3z)\eta(6z))^{2}}
		\]
		is a multiple of $H(\tau)$, and this relation was used to find formulas for $f(q)$ and $\omega(q)$ in terms of traces of CM values of $F$. If we let $\tilde{F}$ be a harmonic Maass form of weight $2$ for $\Gamma_{0}(6)$ which maps to $F$ under $\xi_{2}$, then the differential equation from Proposition~\ref{proposition xi diagram} tells us that the Shintani theta lift $\ISh_{1}(\tilde{F},\tau)$ yields a sesquiharmonic Maass form which is a $\xi_{3/2}$-preimage of a multiple of $H(\tau)$. Its Fourier coefficients of the holomorphic part are given by traces of regularized cycle integrals of $\tilde{F}$.
		
		In her master's thesis, Kupka \cite{kupka} constructed (completions of vector valued versions of) Ramanujan's mock theta functions of order $5$ and $7$ as Millson lifts of weakly holomorphic modular functions, so the Shintani theta lift can immediately be used to construct $\xi_{3/2}$-preimages of these mock theta functions as well.

\subsection{Outline of the work}

The paper is organized as follows.

We explain the basic setup of the work in Section~2. We realize the complex upper half-plane as the Grassmannian of positive definite lines in a rational quadratic space $V$ of signature $(1,2)$, and we define CM points and geodesics in this model. Further, we recall the definition of vector valued harmonic Maass forms for the Weil representation associated to an even lattice.

In Section~3 we define regularized cycle integrals of harmonic Maass forms of positive even weight, and we give an alternative formula as a line integral involving polygamma functions and Bernoulli polynomials.

In Section~4 we recall the definition and the basic properties of the Millson and the Shintani theta functions, i.e., their transformation behaviour, their growth at the cusps, and their relations by differential operators.

The corresponding regularized Millson and Shintani theta lifts are investigated in Section~5. We show that the theta lifts yield real analytic (often harmonic) automorphic forms which are related by the $\xi$-operator.

Section 6 is the technical heart of the work. We state and compute the Fourier expansion of the Shintani theta lift. To this end, we introduce the function $\eta(Q,\tau,z)$ and use it to reduce the computation of the Fourier coefficients of the Shintani theta lift to the computation of several boundary integrals, which we then explicitly evaluate.

\section{Preliminaries}

In order to treat the Shintani theta lift of harmonic Maass forms for arbitrary congruence subgroups we use an orthogonal model of the upper-half plane and we work with vector valued harmonic Maass forms for the Weil representation of an even lattice of signature $(1,2)$. Here we briefly explain the necessary background, following the expositions of \cite{brfu06,bif}.

\subsection{The Grassmannian model of the upper half-plane}

For a positive integer $N$ we consider the rational quadratic space $V$ of signature $(1,2)$ given by the set of all rational traceless $2$ by $2$ matrices

with the quadratic form $Q(X)=N\text{det}(X)$. The associated bilinear form is $(X,Y)=-N\text{tr}(XY)$ for $X,Y \in V$. The group $\SL_2(\Q)$ acts as isometries on $V$ by $\gamma X :=\gamma X \gamma^{-1}$. 

For $z = x+iy \in \mathbb{H}$ we let $g_{z} = \left(\begin{smallmatrix}\sqrt{y} & x/\sqrt{y} \\ 0 & 1/\sqrt{y} \end{smallmatrix} \right) \in \SL_{2}(\R)$, such that $g_{z}i = z$. The vectors
\begin{align*}
	X_{1}(z) &= \frac{1}{\sqrt{2N}y}\begin{pmatrix}-x & x^{2} + y^{2} \\ -1 & x \end{pmatrix} = g_{z}\left(\frac{1}{\sqrt{2N}}\begin{pmatrix}0 & 1 \\ -1 & 0 \end{pmatrix}\right),\\
	X_{2}(z) &= \frac{1}{\sqrt{2N}y}\begin{pmatrix}x & -x^{2} + y^{2} \\ 1 & -x \end{pmatrix} = g_{z}\left(\frac{1}{\sqrt{2N}}\begin{pmatrix}0 & 1 \\ 1 & 0 \end{pmatrix}\right), \\
	X_{3}(z) &= \frac{1}{\sqrt{2N}y}\begin{pmatrix}y & -2xy \\ 0 & -y \end{pmatrix} = g_{z}\left(\frac{1}{\sqrt{2N}}\begin{pmatrix}1 & 0 \\ 0 & -1 \end{pmatrix}\right),
\end{align*}
form an orthogonal basis of $V(\R) = V \otimes \R$ with 
\[(
X_{1}(z),X_{1}(z)) = 1\quad\text{and} \quad(X_{2}(z),X_{2}(z)) = (X_{3}(z),X_{3}(z)) = -1.
\] 
We endow $V$ with the orientation given by this basis.
We let $D$ be the Grassmannian of lines in $V(\R)$ on which the quadratic form $Q$ is positive definite
and we identify $D$ with the complex upper half-plane $\H$ by associating to $z \in \H$ the positive line generated by $X_{1}(z)$.

The group $\SL_{2}(\R)$ acts on $\H$ by fractional linear transformations and
the identification above is $\SL_{2}(\R)$-equivariant, that is, $\gamma X_{1}(z)=X_{1}(\gamma z)$ for $\gamma \in \SL_{2}(\R)$ and $z \in \H$.

Throughout we let $L\subset V$ be an even lattice and $\Gamma$ a congruence subgroup of $\SL_{2}(\Z)$ that takes $L$ to itself and acts trivially on the discriminant group $L'/L$. For $h \in L'/L$ and $m \in \Q$ we let
\[
L_{m,h} = \{X \in L+h: Q(X) = m\}.
\]
The group $\Gamma$ acts on $L_{m,h}$, with finitely many orbits if $m \neq 0$.

\subsection{Cusps} 

We identify the set of isotropic lines $\mathrm{Iso}(V)$ in $V$
with $P^1(\Q)=\Q \cup \left\{ \infty\right\}$ via
\[
\psi: P^1(\Q) \rightarrow \mathrm{Iso}(V), \quad \psi((\alpha:\beta))
 = \mathrm{span}\left(\begin{pmatrix} \alpha\beta &\alpha^2 \\  -\beta^2 & -\alpha\beta \end{pmatrix}\right).
\]
The map $\psi$ is a bijection and $\psi(\gamma(\alpha:\beta))=\gamma\psi((\alpha:\beta))$ for $\gamma \in \SL_{2}(\Q)$. In particular, the cusps of $\Gamma$ can be identified with the $\Gamma$-classes of $\Iso(V)$. The line $\ell_\infty = \psi(\infty)$ is spanned by 
$X_\infty=\left(\begin{smallmatrix}0 & 1 \\ 0 & 0\end{smallmatrix}\right)$. 
For $\ell \in \mathrm{Iso}(V)$ we pick $\sigma_{\ell} \in\SL_2(\Z)$ 
such that $\sigma_{\ell}\ell_\infty=\ell$. An element of $\ell$ will be called positively oriented if it is a positive multiple of $\sigma_{\ell}X_{\infty}$. We let $\Gamma_{\ell}$ be the stabilizer of $\ell$ in $\Gamma$. Then $\sigma_{\ell}^{-1}\overline{\Gamma}_{\ell}\sigma_{\ell}$ is generated by $\left(\begin{smallmatrix}1 & \alpha_{\ell} \\ 0 & 1 \end{smallmatrix}\right)$ for some $\alpha_\ell \in \Q_{>0}$ which we call the width of the cusp $\ell$. For each $\ell$, there is a $\beta_{\ell} \in \Q_{>0}$ such that $\left(\begin{smallmatrix}0 & \beta_{\ell}  \\ 0 & 0\end{smallmatrix}\right)$ is a primitive element of $\ell_{\infty}\cap \sigma_{\ell}^{-1}L$. We write $\varepsilon_{\ell} = \alpha_{\ell}/\beta_{\ell}$. The quantities $\alpha_{\ell},\beta_{\ell}$ and $\varepsilon_{\ell}$ only depend on the $\Gamma$-class of $\ell$.

We let $\calF_{T} = \{z \in \calF: y \leq T\}$ be a truncated version of the standard fundamental domain $\calF = \{z \in \H: |x| \leq 1/2, |z| \geq 1\}$ for $\SL_{2}(\Z) \backslash \H$, and we define a truncated fundamental domain for $\Gamma \backslash \H$ by
\[
\calF(\Gamma)_{T} = \bigcup_{\ell \in \Gamma \setminus \Iso(V)}\sigma_{\ell}\calF^{\alpha_{\ell}}_{T}, \qquad \calF_{T}^{\alpha_{\ell}} = \bigcup_{j=0}^{\alpha_{\ell}-1}\begin{pmatrix} 1 & j \\ 0 & 1 \end{pmatrix}\calF_{T}.
\]
By letting $T$ go to $\infty$, we obtain a fundamental domain $\calF(\Gamma)$ for $\Gamma \backslash \H$.

\subsection{CM points and geodesics}

For $z =x+iy\in \H$ we define polynomials in $X = \left( \begin{smallmatrix}x_{2} & x_{1} \\ x_{3} & -x_{2}\end{smallmatrix}\right) \in V(\R)$ by
\begin{align*}
p_{z}(X) &= \sqrt{2}(X,X_{1}(z)) = -\frac{\sqrt{N}}{y}(x_{3}|z|^{2}-2x_{2}x -x_{1}), \\
Q_{X}(z) &= \sqrt{2N}y(X,X_{2}(z)+iX_{3}(z)) = N(x_{3}z^{2}-2x_{2}z-x_{1}), \\
R(X,z)&= \frac{1}{2}p^{2}_{z}(X)-(X,X) = \frac{1}{2Ny^{2}}|Q_{X}(z)|^{2}.
\end{align*}
For $\gamma \in \SL_{2}(\R)$ we have the transformation rules
\begin{align}\label{zTransformationRules}
p_{\gamma z}(X) = p_{z}(\gamma^{-1}X), \quad Q_{X}(\gamma z) = j(\gamma,z)^{-2}Q_{\gamma^{-1}X}(z), \quad R(X,\gamma z) = R(\gamma^{-1}X,z),
\end{align}
which can be verified by a direct calculation.

For $X \in V$ with $Q(X) > 0$ the CM point $z_{X} \in \H$ associated to $X$ is defined as the point corresponding to the positive line in the Grassmannian $D$ spanned by $X$. Equivalently, $z_{X}$ is the unique root of $Q_{X}(z)$ in $\H$. We use the same symbol $z_{X}$ for its orbit in the modular curve $\Gamma \backslash \H$. Note that the stabilizer $\overline{\Gamma}_{X}$ of $X$ in $\overline{\Gamma}$ is finite. 

For $X \in V$ with $Q(X) < 0$ the geodesic $c_{X} \subseteq \H$ associated to $X$ is defined as the set of all points corresponding to the positive lines in $D$ which are orthogonal to $X$. Equivalently, $c_{X}$ is the set of all $z \in \H$ with $p_{z}(X) = 0$. We orient $c_{X}$ as follows. There is some $g \in \SL_{2}(\R)$ such that $g^{-1}X = \sqrt{|Q(X)|/N}\left(\begin{smallmatrix}-1 & 0 \\ 0 & 1 \end{smallmatrix}\right)$, so $g^{-1}c_{X} = i\R_{>0}$ is the positive imaginary axis. If we move along $i\R_{>0}$ from $0$ to $i\infty$ this puts an orientation on $c_{X} = g (i\R_{>0})$, which is in fact independent of the choice of $g$. 
We write $c(X) = \Gamma_{X} \backslash c_{X}$ for the image of $c_{X}$ in $\Gamma \backslash \H$. 

Let $X \in V$ with $Q(X) < 0$. If $|Q(X)|/N \in \Q$ is not a square then the stabilizer $\overline{\Gamma}_{X}$ is infinite cyclic, and $c(X)$ is a closed geodesic in $\Gamma \backslash \H$. If $|Q(X)|/N$ is a rational square, then $\overline{\Gamma}_{X}$ is trivial, and $c(X)$ is an infinite geodesic. In the latter case, $X$ is orthogonal to two isotropic lines $\ell_{X},\tilde{\ell}_{X} \in \Iso(V)$, which we can tell apart by requiring that if we write $\ell_{X} = \spann(Y),\tilde{\ell}_{X} = \spann(\tilde{Y})$ with $Y$ and $\tilde{Y}$ positively oriented, then $(X,Y,\tilde{Y})$ is a positively oriented basis for $V$. Note that $\tilde{\ell}_{X} = \ell_{-X}$. We have
\[
\sigma_{\ell_{X}}^{-1}X = \sqrt{\frac{|Q(X)|}{N}}\begin{pmatrix}-1 & 2r_{\ell_{X}} \\ 0 & 1 \end{pmatrix}
\]
for some $r_{\ell_{X}} \in \Q$. The geodesic $c_{X}$ in $\H$ is then explicitly given by
\[
c_{X} = \sigma_{\ell_{X}}\{z \in \H: \Re(z) = r_{\ell_{X}}\}.
\] 
Note that two different choices of $\sigma_{\ell_{X}} \in \SL_{2}(\Z)$ with $\sigma_{\ell_{X}}\infty = \ell_{X}$ give two numbers $r_{\ell_{X}}$ which differ by a multiple of the width $\alpha_{\ell_{X}}$. Therefore we call the residue class of $r_{\ell_{X}}$ mod $\alpha_{\ell_{X}}$ the real part of the geodesic $c(X)$.

\subsection{The Weil representation} We let $\C[L'/L]$ be the group ring of $L'/L$, generated by the formal basis vectors $\e_h$ for $h \in L'/L$ and equipped with the scalar product $\langle \e_{h} ,\e_{h'} \rangle = \delta_{h,h'}$, which is conjugate-linear in the second variable. By $\Mp_{2}(\Z)$ we denote the integral metaplectic group consisting of pairs $(\gamma, \phi)$, where $\gamma = \left\{\left(\begin{smallmatrix}a & b \\ c & d \end{smallmatrix} \right) \in \SL_2(\Z)\right\}$ and $\phi:\H\rightarrow \C$
is a holomorphic function with $\phi(\tau)^{2}=c\tau+d$. The Weil representation $\rho_{L}$ of the metaplectic group $\Mp_{2}(\Z)$ is the unitary representation defined for the generators $S=\left(\left(\begin{smallmatrix}0 & -1 \\ 1 & 0 \end{smallmatrix} \right),\sqrt{\tau}\right)$ and $T=\left(\left(\begin{smallmatrix}1 & 1 \\ 0 & 1 \end{smallmatrix} \right), 1\right)$ of $\Mp_2(\Z)$ and $h \in \C[L'/L]$ by the formulas
\begin{align*}
 \rho_{L}(T) \e_h &= e(Q(h)) \e_h,\qquad  \rho_{L}(S) \e_h = \frac{\sqrt{i}}{\sqrt{\abs{L'/L}}}\sum_{h' \in L'/L} e(-(h',h)) \e_{h'},
\end{align*}
where $e(a):=e^{2\pi i a}$. The dual Weil representation will be denoted by $\overline{\rho}_{L}$.

 \subsection{Harmonic Maass forms} A smooth function $G: \H \to \C$ is called a harmonic Maass form of weight $\kappa \in \Z$ for $\Gamma$ if it is annihilated by the Laplace operator $\Delta_{\kappa}$ defined in \eqref{eq Laplace operator}, transforms like a modular form of weight $\kappa$ for $\Gamma$, and is a most of linear exponential growth at the cusps of $\Gamma$. Such a form has a Fourier expansion as in \eqref{eq Fourier expansion harmonic Maass form} at each cusp of $\Gamma$. We denote the space of harmonic Maass forms of weight $\kappa \in \Z$ for $\Gamma$ by $H_{\kappa}(\Gamma)$, and we let $H_{\kappa}^{+}(\Gamma)$ be the subspace of those forms which map to a cusp form under $\xi_{\kappa}$. Further, we let $M^{\text{!}}_{\kappa}(\Gamma)$ be the subspace of weakly holomorphic modular forms, consisting of those harmonic Maass forms which are holomorphic on $\H$, and we let $M_{\kappa}(\Gamma)$ and $S_{\kappa}(\Gamma)$ denote the spaces of holomorphic modular forms and cusp forms of weight $\kappa$ for $\Gamma$. 
 
 Vector valued harmonic Maass forms of half-integral weight $\kappa \in \frac{1}{2}+\Z$ for the Weil representation $\rho_{L}$ are defined to be smooth functions $f : \H \to \C[L'/L]$ which are annihilated (component-wise) by $\Delta_{\kappa}$, transform under $(\gamma,\phi)\in\Mp_{2}(\Z)$ as $f(\gamma\tau) = \phi(\tau)^{2\kappa}\rho_{L}(\gamma,\phi)f(\tau)$, and are at most of linear exponential growth as $v \to \infty$, uniformly in $u$. The corresponding spaces of vector valued harmonic Maass forms and modular forms are denoted by $H_{\kappa,\rho_{L}}, H_{\kappa,\rho_{L}}^{+}, M_{\kappa,\rho_{L}}^{!}, M_{\kappa,\rho_{L}}$ and $S_{\kappa,\rho_{L}}$. Harmonic Maass forms for the dual Weil representation are defined analogously.

\section{Regularized cycle integrals of harmonic Maass forms}\label{section regularized cycle integrals}

In this section we define regularized cycle integrals of harmonic Maass forms in $H_{2k+2}(\Gamma)$ along infinite geodesics, extending the results of \cite{bif}, Section 3.3. Let $X \in V$ with $Q(X) = m < 0$ such that $|m|/4N$ is a square. Then the stabilizer $\overline{\Gamma}_{X}$ is trivial and the geodesic $c(X)$ is infinite in $\Gamma \backslash \H$. Further, $X$ is orthogonal to two isotropic lines $\ell_{X},\tilde{\ell}_{X} \in \Iso(V)$.

Pick some number $c_{+} > 0$. If $G \in S_{2k+2}(\Gamma)$ is a cusp form, then the cycle integral along $c(X)$ converges, and if we write $G_{\ell} = G|_{2k+2}\sigma_{\ell} = \sum_{n > 0}a_{\ell}(n)e(nz)$ for the expansion of $G$ at the cusp $\ell$, then the cycle integral of $G$ along $c(X)$ can explicitly be evaluated as
\begin{align*}
&\frac{1}{\big( 2\sqrt{|m|N}i\big)^{k}i }\int_{c(X)}G(z)Q_{X}^{k}(z)dz \\
&= \left(\int_{c_{+}}^{\infty}G_{\ell_{X}}(r_{\ell_{X}} + iy)y^{k}dy + (-1)^{k+1}\int_{c_{-}}^{\infty}G_{\ell_{-X}}(r_{\ell_{-X}}+iy)y^{k}dy\right) \\
&= \bigg(\sum_{n > 0}a_{\ell_{X}}(n)e^{2\pi i nr_{\ell_{X}}}\frac{\Gamma(k+1,2\pi nc_{+})}{(2\pi n)^{k+1}}  + (-1)^{k+1}\sum_{n > 0}a_{\ell_{-X}}(n)e^{2\pi i n r_{\ell_{-X}}}\frac{\Gamma(k+1,2\pi nc_{-})}{(2\pi n)^{k+1}}\bigg),
\end{align*}
where $c_{-} = \Im\big(\sigma_{\ell_{-X}}^{-1}(r_{\ell_{X}}+ic_{+})\big) = 1/c_{+}b^{2}$ if $r_{\ell_{X}} = a/b$ with coprime $a,b \in \Z$.

Now, if $G \in M_{2k+2}^{!}(\Gamma)$ is a weakly holomorphic modular form whose constant coefficients vanish at all cusps, then the cycle integral along $c(X)$ typically diverges, but the right-hand side of the above formula, with the sums running over all $n \neq 0$, still makes sense. This regularization of the cycle integral was considered before in \cite{brifrike, briguka2, briguka}, and it was used to construct the Shintani lift and regularized $L$-values of weakly holomorphic modular forms. Using a similar idea, regularized cycle integrals of harmonic Maass forms $G \in H_{0}^{+}(\Gamma)$ were defined in \cite{bif}, and they were shown to appear as Fourier coefficients of harmonic Maass forms of weight $1/2$. We now define the regularized cycle integral of $G \in H_{2k+2}(\Gamma)$. To this end, we first determine an antiderivative of $G(iy)y^{k}$.

\begin{lem}\label{lemma regularized cycle integral}
	Let $G \in H_{2k+2}(\Gamma)$ and let $a^{\pm}(n)$ denote the Fourier coefficients of $G$. For $k > 0$ we have
	\begin{align*}
	&\int G(iy)y^{k}dy = a^{+}(0)\frac{y^{k+1}}{k+1}  - \sum_{n \neq 0}a^{+}(n)\frac{\Gamma(k+1,2\pi n y)}{(2\pi n )^{k+1}} - a^{-}(0)\frac{y^{-k}}{k} \\
		&\quad - \sum_{n \neq 0}a^{-}(n)\bigg( \E_{2k+2}(4\pi ny)\frac{\Gamma(k+1,2\pi n y)}{(2\pi n)^{k+1}}-\frac{2^{-2k-1}k!}{(2\pi n)^{k+1}}\sum_{j=0}^{k} \frac{(-1)^{j}}{j!}\E_{2k+2-j}(2\pi ny)\bigg).
	\end{align*}
	For $k = 0$ the same formula holds, but $\frac{y^{-k}}{k}$ has to be replaced by $-\log(y)$.
\end{lem}

\begin{proof}
	Plugging in the Fourier expansion of $G$ we get
	\begin{align*}
	\int G(iy)y^{k}dy &= \int\bigg(a^{+}(0)+\sum_{n \neq 0}a^{+}(n)e^{-2\pi n y}\\
	&\qquad\qquad +a^{-}(0)y^{-2k-1}+\sum_{n \neq 0}a^{-}(n)\E_{2k+2}(4\pi n y)e^{-2\pi n y}\bigg)y^{k}dy.
	\end{align*}
	The integrals corresponding to the coefficients $a^{+}(0), a^{-}(0)$ and $a^{+}(n), n\neq 0$, can be evaluated as
	\begin{align*}
	\int y^{k}dy =\frac{y^{k+1}}{k+1}, \qquad 
	\int y^{-k-1}dy = \begin{dcases}
	\log(y), & \text{if } k =0, \\
	-\frac{y^{-k}}{k}, & \text{if } k >0,
	\end{dcases}
	\end{align*}
	and
	\begin{align*}
	 \int e^{-2\pi n y}y^{k}dy = -\frac{\Gamma(k+1,2\pi n y)}{(2\pi n)^{k+1}}.
	\end{align*}
	In the integral corresponding to $a^{-}(n), n \neq 0,$ we first use integration by parts to obtain
	\begin{align*}
	\int \E_{2k+2}(4\pi ny)e^{-2\pi n y}y^{k}dy &= -\E_{2k+2}(4\pi ny)\frac{\Gamma(k+1,2\pi n y)}{(2\pi n)^{k+1}} \\
	&\quad + \frac{4\pi n}{(2\pi n)^{k+1}}\int (-4\pi n y)^{-2k-2}e^{4\pi n y}\Gamma(k+1,2\pi n y)dy.
	\end{align*}
	In the integral on the right-hand side we plug in the formula $\Gamma(k+1,x) = k!e^{-x}\sum_{j=0}^{k}\frac{x^{j}}{j!}$ to get
	\begin{align*}
	\frac{4\pi n}{(2\pi n)^{k+1}}\int(-4\pi n y)^{-2k-2}e^{4\pi n y}\Gamma(k+1,2\pi n y)dy 
	&= \frac{2^{-2k-1} k!}{(2\pi n)^{k+1}}\sum_{j=0}^{k}\frac{(-1)^{j}}{j!}\E_{2k+2-j}(2\pi n y).
	\end{align*}
	This finishes the proof.
\end{proof}

The above lemma motivates the following definition.

\begin{dfn}
	Let $X \in V$ with $Q(X) = m < 0$ such that $|m|/4N$ is a square, and let $G \in H_{2k+2}(\Gamma)$. Let $T_{+},T_{-} > 0$ be arbitrary. For $k > 0$ we define the regularized cycle integral of $G$ along $c(X)$ by
	\begin{align*}
		&\frac{1}{\big( 2\sqrt{|m|N}i\big)^{k}i}\int_{c(X)}^{\reg}G(z)Q_{X}^{k}(z)dz =\int_{c_{+}}^{T_{+}}G_{\ell_{X}}(r_{\ell_{X}} + iy)y^{k}dy  \\
		&\qquad- a_{\ell_{X}}^{+}(0)\frac{T_{+}^{k+1}}{k+1}  + \sum_{n \neq 0}a_{\ell_{X}}^{+}(n)e^{2\pi i n r_{\ell_{X}}}\frac{\Gamma(k+1,2\pi n T_{+})}{(2\pi n )^{k+1}}+ a_{\ell_{X}}^{-}(0)\frac{T_{+}^{-k}}{k} \\
		&\qquad + \sum_{n \neq 0}a_{\ell_{X}}^{-}(n)e^{2\pi i n r_{\ell_{X}}}\bigg(\E_{2k+2}(4\pi nT_{+})\frac{\Gamma(k+1,2\pi n T_{+})}{(2\pi n)^{k+1}} \\
		&\qquad \qquad \qquad \qquad \qquad \qquad -\frac{2^{-2k-1}k!}{(2\pi n)^{k+1}}\sum_{j=0}^{k} \frac{(-1)^{j}}{j!}\E_{2k+2-j}(2\pi nT_{+})\bigg) \\
		&\qquad+(-1)^{k+1}\cdot (\text{same expression with $\ell_{X},c_{+},T_{+},$ replaced by $\ell_{-X},c_{-},T_{-}$}),
	\end{align*}

where $c_{-} = \Im\big(\sigma_{\ell_{-X}}^{-1}(r_{\ell_{X}}+ic_{+})\big) = 1/c_{+}b^{2}$ if $r_{\ell_{X}} = a/b$ with coprime $a,b \in \Z$.

	For $k = 0$ the regularized cycle integral is defined by the same formula, but $\frac{T_{\pm}^{-k}}{k}$ has to be replaced by $-\log(T_{\pm})$.
\end{dfn}

\begin{rem}
	Lemma \ref{lemma regularized cycle integral} implies that the regularized cycle integral is well-defined, i.e., independent of the choice of $T_{+}$ and $T_{-}$. If we plug in $T_{+} = c_{+}$ and $T_{-} = c_{-}$ we obtain a formula for the regularized cycle integral which is analogous to the one given above in the case that $G$ is a cusp form.
\end{rem}

In analogy to Theorem 3.2 of \cite{bif}, we give another representation of the regularized cycle integral involving the Bernoulli polynomials $B_{k}(x)$ defined by $te^{xt}/(e^{t}-1) = \sum_{k =0}^{\infty}B_{k}(x)t^{k}/k!$, and the polygamma function $\psi^{(k)}(x) = \frac{d^{k+1}}{dx^{k+1}}\log(\Gamma(x))$.

\begin{prop}\label{prop regularized cycle integrals alternative}
	Let $X \in V$ with $Q(X) = m < 0$ such that $|m|/4N$ is a square, and let $G \in H_{2k+2}(\Gamma)$. Let $T_{+},T_{-} > 0$ be arbitrary. For $k > 0$ we have the formula
	\begin{align*}
	&\frac{1}{\big(2\sqrt{|m|N}i\big)^{k}i}\int_{c(X)}^{\reg}G(z)Q_{X}^{k}(z)dz = \int_{c_{+}}^{T_{+}}G_{\ell_{X}}(iy)y^{k}dy \\
	&+ (-1)^{k}\frac{(i\alpha_{\ell_{X}})^{k+1}}{k+1}\int_{iT_{+}/\alpha_{\ell_{X}}}^{iT_{+}/\alpha_{\ell_{X}}+1}\big(G_{\ell_{X}}(r_{\ell_{X}} + \alpha_{\ell_{X}}z)-a_{\ell_{X}}^{-}(0)y^{-1-2k}\big)B_{k+1}\left(z\right)dz \\
	& - (-1)^{k}(i\alpha_{\ell_{X}})^{-k}\frac{2^{2k-1}k!}{(2k+1)!} \\
	&\quad  \times \overline{\int_{iT_{+}/\alpha_{\ell_{X}}}^{iT_{+}/\alpha_{\ell_{X}}+1}\big((\xi_{2k+2}G_{\ell_{X}})(r_{\ell_{X}}+\alpha_{\ell_{X}}z)-a_{\xi_{2k+2}G,\ell_{X}}(0)\big)\left(\psi^{(k)}\left(z\right)+(-1)^{k}\psi^{(k)}\left(1-z\right)\right)dz} \\
	& - (-1)^{k}(i\alpha_{\ell_{X}})^{-k}\frac{k!}{(2k+1)!}\sum_{d = 0}^{k}\frac{(d+k)!}{(d+1)!} \sum_{j=0}^{k-d}\binom{2k+1}{j} \\
	&\quad\times\overline{\int_{iT_{+}/\alpha_{\ell_{X}}}^{iT_{+}/\alpha_{\ell_{X}}+1}\big((\xi_{2k+2}G_{\ell_{X}})(r_{\ell_{X}}+\alpha_{\ell_{X}}z)-a_{\xi_{2k+2}G,\ell_{X}}(0)\big)B_{d+1}\left(x\right)\left(-iy\right)^{-k-d-1}dz} \\
	&+ a_{\ell_{X}}^{-}(0)\frac{T_{+}^{-k}}{k}\\
	&+(-1)^{k+1}\cdot (\text{same expression with $\ell_{X},c_{+},T_{+}$ replaced by $\ell_{-X},c_{-}, T_{-}$}).
	\end{align*}
	For $k = 0$ the same formula holds, but $\frac{T_{+}^{-k}}{k}$ in the next-to-last line has to be replaced with $-\log(T_{+})$.
\end{prop}

\begin{rem}
	Although the above formula looks very complicated we made the effort to write it down since it is in fact the expression that we obtain in the computation of the Fourier expansion of the Shintani lift of $G \in H_{2k+2}(\Gamma)$. Also note that the formula beautifully simplifies for weakly holomorphic modular forms $G \in M_{2k+2}^{!}(\Gamma)$.
\end{rem}

The proposition can be proved by a straightforward calculation using the identity
\begin{align*}
		\sum_{j=0}^{k}\frac{(-1)^{j}}{j!}\E_{2k+2-j}(y) &= \frac{(-1)^{k+1}2^{2k}k!}{(2k+1)!}\E_{k+1}(y) -  e^{y}\sum_{d=0}^{k}y^{-k-d-1}\frac{(d+k)!}{(2k+1)!}\sum_{j = 0}^{k-d}\binom{2k+1}{j},
		\end{align*}
		(which can most easily be checked by comparing the derivatives of both sides) and the following two lemmas.

\begin{lem}\label{lm:evalB}
		For $n \in \Z$ and $j \geq 1$ we have
		\begin{align*}
		\int_{0}^{1}B_{0}(x)e^{2\pi i n x}dx = \begin{cases}
		1, & n = 0, \\
		0, & n \neq 0.
		\end{cases} \qquad \int_{0}^{1}B_{j}(x)e^{2\pi i n x}dx = \begin{dcases}
		0, & n = 0, \\
		\frac{(-1)^{j+1}j!}{(2\pi in)^{j}}, & n \neq 0.
		\end{dcases}
		\end{align*}
		Further, for $k \in \Z_{\geq 0},y > 0$ and $n \in \Z$ and  we have
		\begin{align*}
		\int_{iy}^{iy+1}B_{k}(z)e^{2\pi i nz}dz = \begin{dcases}
		(iy)^{k}, & \text{if } n =0, \\
		-ki^{k}\frac{\Gamma(k,2\pi n y)}{(2\pi n )^{k}},& \text{if } n \neq 0.
		\end{dcases}
		\end{align*}
	\end{lem}
	
	\begin{proof}
		For $B_{0}(x) = 1$ this is clear, for $B_{1}(x) = x-1/2$ we use integration by parts, and for $j > 1$ we use the absolutely convergent Fourier series $B_{j}(x) = -\frac{j!}{(2\pi i)^{j}}\sum_{k \neq 0}\frac{e^{2\pi i k x}}{k^{j}}$ for $0 < x < 1$. The last formula follows from the others since $B_{k}(z) = \sum_{j=0}^{k}\binom{k}{j}B_{j}(x)(iy)^{k-j}$.
	\end{proof}

	\begin{lem}\label{lm:evalpsi}
	For $k \in \Z_{\geq 0}, y > 0$, and $n \in \Z$ we have
		\begin{align*}
		&\int_{iy}^{iy+1}\left(\psi^{(k)}(z)+(-1)^{k}\psi^{(k)}(1-z)\right)e^{2\pi i n z}dz \\
		&\quad = \begin{cases}
		2\log(y), & \text{if $n= 0$ and $k = 0$},  \\
		2(-1)^{k+1}(k-1)!(iy)^{-k}  , & \text{if $n = 0$ and $k > 0$}, \\
		-2(2\pi i n)^{k}k!\E_{k+1}(-2\pi n y), & \text{if }n \neq 0.
		\end{cases}
		\end{align*}
	\end{lem}

	\begin{proof}
		Since $\psi^{(k)}(x) = \frac{d^{k+1}}{dx^{k+1}}\log(\Gamma(x))$, the integral can immediately be evaluated for $n = 0$. For $n = 0$ and $k > 0$ we additionally use the recurrence relation
		\[
		\psi^{(k)}(z+1) = \psi^{(k)}(z) + \frac{(-1)^{k}k!}{z^{k+1}} \quad (k > 0),
		\]
		see \cite[6.4.6]{abramowitz}, to obtain the stated formula.
		
		For $n \neq 0$ we plug in the expansions
		\[
		\psi(z) = -\gamma + \sum_{n = 0}^{\infty}\left(\frac{1}{n+1}-\frac{1}{n+z}\right), \quad \psi^{(k)}(z) = (-1)^{k+1}k!\sum_{n =0}^{\infty}\frac{1}{(z+n)^{k+1}} \quad (k > 0),
		\]
		see \cite[6.3.16, 6.4.10]{abramowitz}, to find
		\begin{align*}
		&\int_{iy}^{iy+1}\left(\psi^{(k)}(z)+(-1)^{k}\psi^{(k)}(1-z)\right)e^{2\pi i n z}dz \\
		&= (-1)^{k+1}k!\left(\int_{iy}^{iy+\infty}e^{2\pi i nz}\frac{dz}{z^{k+1}}+ (-1)^{k} \overline{\int_{iy}^{iy+\infty}e^{2\pi inz}\frac{dz}{z^{k+1}}}\right).
		\end{align*}
		The formula now follows by a repeated application of integration by parts and the evaluation
		\begin{align*}
		\int_{iy}^{iy+\infty}e^{2\pi i n z}\frac{dz}{z} = \begin{cases}
		-\Ei(-2\pi n y), & n > 0,\\
		-\Ei(-2\pi n y) - i\pi, & n < 0,
		\end{cases}
		\end{align*}
		see \cite[5.1.30/31]{abramowitz}. 
	\end{proof}

\section{Theta functions}\label{sec:thetafunctions}

We now define the theta functions that we use to construct the Millson and the Shintani theta lifts. For $\tau = u + iv,z = x+iy \in \mathbb{H}$ and $k \in \Z_{\geq 0}$ we consider the functions
\begin{align*}
\Theta_{M}(\tau,z) &= v^{k+1}\sum_{h \in L'/L}\sum_{X \in L+h}p_{z}(X)Q^{k}_{X}(\bar{z})e^{-2\pi vR(X,z)}e^{2\pi i Q(X)\tau}\e_{h}, \\
\Theta_{Sh}(\tau,z) &= v^{1/2}\sum_{h \in L'/L}\sum_{X \in L+h}y^{-2k-2}Q^{k+1}_{X}(\bar{z})e^{-2\pi v R(X,z)}e^{2\pi i Q(X)\tau}\e_{h},
\end{align*}
which we call the Millson and the Shintani theta function, respectively. They have been studied in many recent works, for example \cite{brfu04,brfu06, hoevel,bif,crawford}.

\begin{prop}\label{prop:propertiestheta}
	Let $k \in \Z$, $k \geq 0$.
	\begin{enumerate}
		\item The Millson theta function $\Theta_{M}(\tau,z)$ has weight $1/2-k$ in $\tau$ for the representation $\rho_{L}$ and $\overline{\Theta_{M}(\tau,z)}$ has weight $-2k$ in $z$ for $\Gamma$.
		\item The Shintani theta function $\overline{\Theta_{Sh}(\tau,z)}$ has weight $k+3/2$ in $\tau$ for the representation $\overline{\rho}_{L}$ and $\Theta_{Sh}(\tau,z)$ has weight $2k+2$ in $z$ for $\Gamma$.
	\end{enumerate}
\end{prop}

\begin{proof}

	The behaviour in $z$ easily follows from the rules (\ref{zTransformationRules}), and the behaviour in $\tau$ follows from \cite{borcherds}, Theorem 4.1, if we note that the Millson and the Shintani theta functions are essentially the theta functions associated to the polynomials $p_{z}(X)Q_{X}^{k}(\bar{z})$ and $Q_{X}^{k+1}(\bar{z})$, which are homogeneous of degree $(1,k)$ and $(0,k+1)$, respectively.
\end{proof}

We want to investigate the growth of the theta functions at the cusps of $\Gamma$. To describe this in a convenient way, we follow the ideas of \cite[Section 2.2]{bif} and define certain theta functions associated to the cusps. 

For an isotropic line $\ell \in \Iso(V)$ the space $W_{\ell} = \ell^{\perp}/\ell$ is a unary negative definite quadratic space with the quadratic form $Q(X + \ell) := Q(X)$, and
\[
K_{\ell} = (L\cap \ell^{\perp})/(L \cap \ell)
\]
is an even lattice with dual lattice
\[
K'_{\ell} = (L' \cap \ell^{\perp})/(L' \cap \ell).
\]
The vector $X_{\ell}=\sigma_{\ell}.X_{3}(i)$ is a basis of $W_{\ell}$ with $(X_{\ell},X_{\ell}) = -1$, and for $k \in \Z_{\geq 0}$ the polynomial $p_{\ell,k}(X) = (-\sqrt{2N}i(X,X_{\ell}))^{k}$ is homogeneous of degree $(0,k)$. We let $\Theta_{\ell,k}(\tau)$ be the theta function associated to $K_{\ell}$ and $p_{\ell,k}$ as in \cite{borcherds}, Section 4. By \cite[Theorem 4.1]{borcherds} the complex conjugate $\overline{\Theta_{\ell,k}(\tau)}$ is an almost holomorphic modular form of weight $k+1/2$ for the dual Weil representation of $K_{\ell}$. Using \cite[Lemma 5.6]{brhabil}, it gives rise to an almost holomorphic modular form of weight $k+1/2$ for the dual Weil representation $\overline{\rho}_{L}$ of $L$, which we also denote by $\overline{\Theta_{\ell,k}(\tau)}$. For $k = 0$ it is a holomorphic modular form, and for $k = 1$ it is a cusp form. 

Let us denote the Fourier coefficients at $q^{m}$ of the $h$-component of $\overline{\Theta_{\ell,k}(\tau)}$ by $b_{\ell,k}(m,h)$. Then a straightforward calculation shows that $b_{\ell,k}(0,h) = 0$ unless $\ell \cap (L+h)\neq \emptyset$, in which case we have
		\begin{align}\label{eq zero coefficients unary theta functions}
		b_{\ell,k}(0,h) = \frac{(-\sqrt{N}i)^{k}}{(4\pi v)^{k/2}}H_{k}\left(0\right),
		\end{align}
		and for $m > 0$ we have $b_{\ell,k}(m,h) = 0$ unless $m/N$ is a square and there exists a vector $X \in L_{-m,h}$ orthogonal to $\ell$, in which case we have
		\begin{align}\label{eq positive coefficients unary theta functions}
		b_{\ell,k}(m,h) = (\pm 1)^{k}\frac{(-\sqrt{N}i)^{k}}{(4\pi v)^{k/2}}H_{k}\left( 2\sqrt{\pi m v}\right)
		\end{align}
		if $h \neq -h \mod L$, and $1+(-1)^{k}$ times this expression if $h = -h \mod L$. Here $H_{k}(x) = (-1)^{k}e^{x^{2}}\frac{d^{k}}{dx^{k}}e^{-x^{2}}$ denotes the $k$-th Hermite polynomial, and the sign is $+1$ if $\ell = \ell_{X}$, and $-1$ if $\ell = \ell_{-X}$. Further, for $m < 0$ we have $b_{\ell,k}(m,h) = 0$.

\begin{prop}\label{prop:growththeta}
  Let $k\geq 0$ and let $\ell$ be a cusp of $\Gamma$ .
  \begin{enumerate}
  \item For the Millson theta function we have
   \begin{align*}
 j(\sigma_{\ell},\bar{z})^{2k}\Theta_{M}(\tau,\sigma_\ell z)&=
  -y^{k+1}\frac{k}{2\pi\beta_{\ell}}v^{k-1/2}\Theta_{\ell,k-1}(\tau) + O(e^{-Cy^{2}}),
  \end{align*}
 as $y \rightarrow\infty$, uniformly in $x$, for some constant $C>0$, where we set $\overline{\Theta_{\ell,-1}(\tau)} = 0$.
  \item For the Shintani theta function we have
  \begin{align*}
  j(\sigma_{\ell},z)^{-2k-2}\Theta_{Sh}(\tau,\sigma_{\ell}z) &= y^{-k}\frac{1}{\sqrt{N}\beta_{\ell}}\Theta_{\ell,k+1}(\tau) + O(e^{-Cy^{2}}),
 \end{align*}
 as $y \rightarrow\infty$, uniformly in $x$, for some constant $C>0$.
 \end{enumerate}
 Moreover, all of the partial derivates of the functions hidden in the $O$-notation are square exponentially decreasing as $y \to \infty$.
\end{prop}

\begin{proof}
	Using the rules (\ref{zTransformationRules}) we see that $j(\sigma_{\ell},\bar{z})^{2k}\Theta_{M}(\tau,\sigma_\ell z)$ equals the Millson theta function associated to the lattice $\sigma_{\ell}^{-1}L$, and similarly for the Shintani theta function. Hence, we can equivalently estimate the growth of the theta functions for the lattice $\sigma_{\ell}^{-1}L$ at the cusp $\infty$. The result now follows from Theorem 5.2 in \cite{borcherds} applied to the lattice $\sigma_{\ell}^{-1}L$ and the primitive isotropic vector $\left(\begin{smallmatrix}0 & \beta_{\ell} \\ 0 & 0\end{smallmatrix}\right) \in  \ell_{\infty}\cap\sigma_{\ell}^{-1}L$.
\end{proof}

The theta functions we just defined satisfy some interesting differential equations. All of the following identities are well known and can be checked by a direct computation using the rules
\begin{align}\label{eq:diffpQR}
\frac{\partial}{\partial z} \left(y^{-2}Q_{X}(z)\right)=-i\sqrt{N}y^{-2}p_z(X),  \qquad \frac{\partial}{\partial z} p_{z}(X) =-\frac{i}{2\sqrt{N}}y^{-2}Q_{X}(\bar{z}).
\end{align}

\begin{lem}\label{lm:iddelta}
	For $k \geq 0,$ we have
\begin{align*}
 4\Delta_{1/2-k,\tau}\Theta_{M}(\tau,z) &= \overline{\Delta_{-2k,z}\overline{\Theta_{M}(\tau,z)}}, \\
 4\overline{\Delta_{k+3/2,\tau}\overline{\Theta_{Sh}(\tau,z)}} &= \Delta_{2k+2,z}\Theta_{Sh}(\tau,z).
\end{align*}
\end{lem}

The Millson and the Shintani theta functions are related by the following identities.

\begin{lem}\label{lm:relshinmillson}
	For $k \geq 0$ we have
\begin{align*}
\xi_{1/2-k,\tau} \Theta_{M}(\tau,z) &= \frac{1}{2\sqrt{N}}\xi_{2k+2,z}\Theta_{Sh}(\tau,z), \\
\xi_{3/2+k,\tau} \overline{\Theta_{Sh}(\tau,z)}&=\frac{\sqrt{N}}{2} \xi_{-2k,z}\overline{\Theta_{M}(\tau,z)}.
\end{align*}
\end{lem}

\section{The Shintani and the Millson theta lift}\label{section theta lift}

If $G \in S_{2k+2}(\Gamma)$ is a cusp form, the usual Shintani theta lift of $G$ is given by \begin{align*}
\ISh(G,\tau) &= \int_{\calF(\Gamma)}G(z)\overline{\Theta_{Sh}(\tau,z)}y^{2k+2}d\mu(z)\\
&= \sum_{\ell \in \Gamma \backslash \Iso(V)}\int_{\calF^{\alpha_{\ell}}}(G|_{2k+2}\sigma_{\ell})(z)\overline{(\Theta_{Sh}|_{2k+2}\sigma_{\ell})(\tau,z)}y^{2k+2}d\mu(z).
\end{align*}
For $G \in H_{2k+2}(\Gamma)$ the integrals do usually not converge. Following ideas from \cite{bif}, we make the following definition.

\begin{dfn} We define the regularized Shintani theta lift of a harmonic Maass form $G \in H_{2k+2}(\Gamma)$ by
\[
\ISh(G,\tau) = \CT_{s = 0}\left[\ISh(G,\tau,s)\right],
\]
where
\[
\ISh(G,\tau,s) = \sum_{\ell \in \Gamma \backslash \Iso(V)}\lim_{T \to \infty}\int_{\calF_{T}^{\alpha_{\ell}}}(G|_{2k+2}\sigma_{\ell})(z)\overline{(\Theta_{Sh}|_{2k+2}\sigma_{\ell})(\tau,z)}y^{2k+2-s}d\mu(z),
\]
and $\CT_{s = 0}\left[\ISh(G,\tau,s)\right]$ denotes the constant term in the Laurent expansion at $s=0$ of the analytic continuation of $\ISh(G,\tau,s)$.
\end{dfn}

\begin{prop}\label{proposition convergence shintani lift}
	For $G \in H_{2k+2}(\Gamma)$ the function
	\[
	\ISh(G,\tau,s) - \frac{1}{\sqrt{N}}\sum_{\ell \in \Gamma \backslash \Iso(V)}\varepsilon_{\ell}\overline{\Theta_{\ell,k+1}(\tau)}\left(\frac{a_{\ell}^{+}(0)}{s-k-1} + \frac{a_{\ell}^{-}(0)}{s+k} \right)
	\]
	converges locally uniformly for $\Re(s) > k+1$ and has a holomorphic continuation to $\C$, which is smooth in $\tau$. In particular, the regularized Shintani lift $\ISh(G,\tau)$ is a smooth function transforming like a modular form of weight $k+3/2$ for $\bar{\rho}_{L}$. Further, for $k > 0$ we have the formula
	\begin{align*}
	\ISh(G,\tau)	&=\lim_{T \to \infty} \bigg[\int_{\calF(\Gamma)_{T}}G(z)\overline{\Theta_{Sh}(\tau,z)}y^{2k+2}d\mu(z) \\
	&\qquad \qquad  - \frac{1}{\sqrt{N}}\sum_{\ell \in \Gamma \backslash \Iso(V)}\varepsilon_{\ell}\overline{\Theta_{\ell,k+1}(\tau)}\left(a_{\ell}^{+}(0)\frac{T^{k+1}}{k+1}-a_{\ell}^{-}(0)\frac{T^{-k}}{k} \right)\bigg],
	\end{align*}
	and for $k = 0$ the same formula holds with $\frac{T^{-k}}{k}$ replaced by $-\log(T)$.
\end{prop}

\begin{proof}
	We first write
	\begin{align*}
	&\ISh(G,\tau,s) \\
	&= \sum_{\ell \in \Gamma \backslash \Iso(V)}\lim_{T \to \infty}\int_{\calF_{T}^{\alpha_{\ell}}}(G|_{2k+2}\sigma_{\ell})(z)\overline{\left((\Theta_{Sh}|_{2k+2}\sigma_{\ell})(\tau,z)-\frac{y^{-k}}{\sqrt{N}\beta_{\ell}}\Theta_{\ell,k+1}(\tau)\right)}y^{2k+2-s}d\mu(z) \\
	& \qquad + \sum_{\ell \in \Gamma \backslash \Iso(V)}\frac{1}{\sqrt{N}\beta_{\ell}}\overline{\Theta_{\ell,k+1}(\tau)}\lim_{T \to \infty}\int_{\calF_{T}^{\alpha_{\ell}}}(G|_{2k+2}\sigma_{\ell})(z)y^{k+2-s}d\mu(z).
	\end{align*}
	According to the growth estimates for the Shintani theta function, the first integral defines a holomorphic function for all $s\in \C$. In the second line, we write
	\begin{align*}
	\lim_{T \to \infty}\int_{\calF_{T}^{\alpha_{\ell}}}(G|_{2k+2}\sigma_{\ell})(z)y^{k+2-s}d\mu(z) &= \int_{\calF_{1}^{\alpha_{\ell}}}(G|_{2k+2}\sigma_{\ell})(z)y^{k+2-s}d\mu(z) \\
	&\quad + \int_{1}^{\infty}\int_{0}^{\alpha_{\ell}}(G|_{2k+2}\sigma_{\ell})(z)y^{k-s} dx\,dy.
	\end{align*}
	The integral over the compact domain $\mathcal{F}_{1}^{\alpha_{\ell}}$ on the right-hand side is holomorphic for all $s \in \C$. The integral over $x$ in the second line picks out the zero coefficient of $G|_{2k+2}\sigma_{\ell}$, so we obtain for $\Re(s) > k+1$
	\begin{align*}
	\int_{1}^{\infty}\int_{0}^{\alpha_{\ell}}(G|_{2k+2}\sigma_{\ell})(z)y^{k-s} dx\,dy &= \alpha_{\ell}\int_{1}^{\infty}\left(a_{\ell}^{+}(0) + a_{\ell}^{-}(0)y^{1-(2k+2)} \right) y^{k-s} dy \\
	&= \alpha_{\ell}\left(\frac{a_{\ell}^{+}(0)}{s-k-1} + \frac{a_{\ell}^{-}(0)}{s+k} \right).
	\end{align*}
	This shows that the regularized integral exists. We have just proved that the summand corresponding to the cusp $\ell$ in $\ISh(G,\tau)$ is given by
	\begin{align*}
	&\lim_{T \to \infty}\int_{\calF_{T}^{\alpha_{\ell}}}(G|_{2k+2}\sigma_{\ell})(z)\overline{\left((\Theta_{Sh}|_{2k+2}\sigma_{\ell})(\tau,z)-\frac{y^{-k}}{\sqrt{N}\beta_{\ell}}\Theta_{\ell,k+1}(\tau)\right)}y^{2k+2}d\mu(z) \\
	&\quad+\frac{1}{\sqrt{N}\beta_{\ell}}\overline{\Theta_{\ell,k+1}(\tau)}\left(\int_{\calF_{1}^{\alpha_{\ell}}}(G|_{2k+2}\sigma_{\ell})(z)y^{k+2}d\mu(z) + \alpha_{\ell}\CT_{s = 0}\left(\frac{a_{\ell}^{+}(0)}{s-k-1} + \frac{a_{\ell}^{-}(0)}{s+k} \right)\right).
	\end{align*}
	For $k > 0$ we have
	\begin{align*}
\CT_{s = 0}\left(\frac{a_{\ell}^{+}(0)}{s-k-1} + \frac{a_{\ell}^{-}(0)}{s+k} \right)&= \lim_{T \to\infty}\left(a_{\ell}^{+}(0)\left(\int_{1}^{T}y^{k}dy-\frac{T^{k+1}}{k+1}\right)\right. \\
	&\qquad \qquad \qquad + \left. a_{\ell}^{-}(0)\left( \int_{1}^{T}y^{-k-1}dy+\frac{T^{-k}}{k}\right) \right),
	\end{align*}
	and for $k = 0$ the same formula holds with $\frac{T^{-k}}{k}$ replaced by $-\log(T)$. This easily implies the stated formula.
\end{proof}

Next, we show that the Shintani theta lift of a harmonic Maass form $G \in H_{2k+2}(\Gamma)$ is related to the Millson theta lift of the weakly holomorphic modular form $\xi_{2k+2}G$ by the $\xi$-operator. In \cite{alfesschwagen} the Millson theta lift $\IM(F,\tau)$ was defined for $F \in H_{-2k}^{+}(\Gamma)$, and it was shown that $\IM(F,\tau)$ defines a harmonic Maass form in $H_{1/2-k,\rho_{L}}^{+}$ which is related to the Shintani lift of the cusp form $\xi_{-2k}F$ by the $\xi$-operator. We generalize the Millson lift to the full space $H_{-2k}(\Gamma)$.

\begin{dfn} We define the regularized Millson theta lift of a harmonic Maass form $F \in H_{-2k}(\Gamma)$ by
\[
\IM(F,\tau) = \CT_{s = 0}\left[\IM(F,\tau,s)\right],
\]
where
\[
\IM(F,\tau,s) = \sum_{\ell \in \Gamma \backslash \Iso(V)}\lim_{T \to \infty}\int_{\calF_{T}^{\alpha_{\ell}}}(F|_{-2k}\sigma_{\ell})(z)\overline{(\overline{\Theta_{M}}|_{-2k}\sigma_{\ell})(\tau,z)}y^{-2k-s}d\mu(z).
\]
\end{dfn}

By similar arguments as in the proof of Proposition \ref{proposition convergence shintani lift} we obtain the following result.

\begin{prop}\label{proposition convergence millson lift}
 
	For $F \in H_{-2k}(\Gamma)$ the function
	\[
	\IM(F,\tau,s) + \frac{k}{2\pi}\sum_{\ell \in \Gamma \backslash \Iso(V)}\varepsilon_{\ell}v^{k-1/2}\Theta_{\ell,k-1}(\tau)\left(\frac{a_{\ell}^{+}(0)}{s+k} + \frac{a_{\ell}^{-}(0)}{s-k-1} \right)
	\]
	converges locally uniformly for $\Re(s) > k+1$ and has a holomorphic continuation to $\C$, which is smooth in $\tau$. In particular, the regularized Millson lift $\IM(F,\tau)$ is a smooth function transforming like a modular form of weight $1/2-k$ for $\rho_{L}$. Further, for $k > 0$ we have the formula
	\begin{align*}
	\IM(F,\tau)	&=\lim_{T \to \infty} \bigg[\int_{\calF(\Gamma)_{T}}F(z)\Theta_{M}(\tau,z)y^{-2k}d\mu(z) \\
	&\qquad \qquad  - \frac{k}{2\pi}\sum_{\ell \in \Gamma \backslash \Iso(V)}\varepsilon_{\ell}v^{k-1/2}\Theta_{\ell,k-1}(\tau)\left(a_{\ell}^{+}(0)\frac{T^{-k}}{k}-a_{\ell}^{-}(0)\frac{T^{k+1}}{k+1} \right)\bigg],
	\end{align*}
	and for $k = 0$ the analogous formula holds with the second line omitted.
\end{prop}

\begin{prop}\label{proposition xi diagram}
	For $G \in H_{2k+2}(\Gamma)$ we have
	\[
	\xi_{3/2+k}\ISh(G,\tau) = -\frac{\sqrt{N}}{2}\IM(\xi_{2k+2}G,\tau),
	\]
	and for $F \in H_{-2k}(\Gamma)$ we have
	\[
\xi_{1/2-k}\IM(F,\tau) = -\frac{1}{2\sqrt{N}}\ISh(\xi_{-2k}F,\tau) + \begin{cases}
\frac{1}{2N}\sum_{\ell \in \Gamma \backslash \Iso(V)}\varepsilon_{\ell}\overline{a_{\ell}^{+}(0)\Theta_{\ell,1}(\tau)}, & \text{if }k = 0, \\
0, & \text{if } k > 0.
\end{cases}
\]

\end{prop}

\begin{proof}
	We only prove the first formula since the second one is very similar. For $k > 0$ we have
	\begin{align*}
	\xi_{3/2+k}\ISh(G,\tau)&=\lim_{T \to \infty} \bigg[\frac{\sqrt{N}}{2}\int_{\calF(\Gamma)_{T}}\overline{G(z)}\xi_{-2k}\overline{\Theta_{M}(\tau,z)}y^{2k+2}d\mu(z) \\
	&\qquad - \frac{1}{\sqrt{N}}\sum_{\ell \in \Gamma \backslash \Iso(V)}\varepsilon_{\ell}\left(\xi_{3/2+k}\overline{\Theta_{\ell,k+1}(\tau)}\right)\left(\overline{a_{\ell}^{+}(0)}\frac{T^{k+1}}{k+1}-\overline{a_{\ell}^{-}(0)}\frac{T^{-k}}{k} \right)\bigg].
	\end{align*}
	Note that for $k = 0$ the function $\overline{\Theta_{\ell,1}(\tau)}$ is a cusp form, so in this case the above formula also holds, but the second line vanishes. Using Stokes' theorem we find
	\begin{align*}
	\xi_{3/2+k}\ISh(G,\tau)&=\lim_{T \to \infty} \bigg[-\frac{\sqrt{N}}{2}\int_{\calF(\Gamma)_{T}}\xi_{2k+2}G(z)\Theta_{M}(\tau,z)y^{-2k}d\mu(z) \\
	&\qquad   - \frac{\sqrt{N}}{2}\int_{\partial \calF(\Gamma)_{T}}\overline{G(z)}\Theta_{M}(\tau,z)d\bar{z} \\
	&\qquad   - \frac{1}{\sqrt{N}}\sum_{\ell \in \Gamma \backslash \Iso(V)}\varepsilon_{\ell}\left(\xi_{3/2+k}\overline{\Theta_{\ell,k+1}(\tau)}\right)\left(\overline{a_{\ell}^{+}(0)}\frac{T^{k+1}}{k+1}-\overline{a_{\ell}^{-}(0)}\frac{T^{-k}}{k} \right)\bigg],
	\end{align*}
	where we again understand that the last line vanishes if $k =0$. By the growth estimate for the Millson theta function we find
	\begin{align*}
	&\int_{\partial \calF(\Gamma)_{T}}\overline{G(z)}\Theta_{M}(\tau,z)d\bar{z}\\
	&= \frac{k}{2\pi}\sum_{\ell \in \Gamma \backslash \Iso(V)}\varepsilon_{\ell}v^{k-1/2}\overline{\Theta_{\ell,k-1}(\tau)}\left(\overline{a_{\ell}^{+}(0)}T^{k+1} + \overline{a_{\ell}^{-}(0)}T^{-k}\right) + O(e^{-\varepsilon T^{2}})
	\end{align*}
	as $T \to \infty$ for some $\varepsilon > 0$. For $k = 0$ we use that $\overline{\Theta_{\ell,-1}(\tau)} = 0$ by definition, and for $k > 0$ we see from the Fourier expansion of $\overline{\Theta_{\ell,k}(\tau)}$ that
	\[
	\xi_{3/2+k}\overline{\Theta_{\ell,k+1}(\tau)} = -\frac{Nk(k+1)}{4\pi}v^{k-1/2}\Theta_{\ell,k-1}(\tau).
	\]
	Thus the boundary integral over $\partial \calF(\Gamma)_{T}$ cancels out with the sum over the theta series as $T \to \infty$, and we obtain the stated formula.
\end{proof}

\begin{prop}\label{proposition Delta}
	For $G \in H_{2k+2}(\Gamma)$ we have
	\begin{align*}
	4\Delta_{3/2+k}\ISh(G,\tau) = \begin{cases}
	-\frac{1}{\sqrt{N}}\sum_{\ell \in \Gamma \backslash \Iso(V)}\varepsilon_{\ell}a_{\ell}^{-}(0)\overline{\Theta_{\ell,1}(\tau)}, & \text{if } k = 0, \\
	0, & \text{if } k > 0,
	\end{cases}
	\end{align*}
	and for $F \in H_{-2k}(\Gamma)$ we have
	\begin{align*}
	\Delta_{1/2-k}\IM(F,\tau) = 0.
	\end{align*}
	In particular, the Shintani and the Millson theta lifts are real analytic on $\H$.
\end{prop}

\begin{proof}
	This can be proved using Stokes' theorem in a similar way as in Section 5.4.1 of \cite{bif}. We will omit the proof for brevity, since the result for the Shintani lift also follows by applying the Laplace operator to the Fourier expansion given below, and the result for the Millson lift easily follows from the one for the Shintani lift if we write $\Delta_{k} = -\xi_{2-k}\xi_{k}$ and use Proposition \ref{proposition xi diagram}.
\end{proof}

\section{The Fourier expansion of the Shintani lift}

	In this section we state the Fourier expansion of the Shintani lift of a harmonic Maass form $G \in H_{2k+2}(\Gamma)$, which is the main result of this work. To this end, we first have to introduce some more notation.
	
	For a $\Gamma$-invariant function $F: \H \to \C$ and $ m \in \Q_{ > 0}, h \in L'/L,$ with $m \in \Z + Q(h)$ we define the trace function
	\begin{align*}
	\tr^{+}(F,m,h) &= \sum_{X \in \Gamma \backslash L_{m,h}^{+}}\frac{F(z_{X})}{|\overline{\Gamma}_{X}|},
	\end{align*}
	where $L_{m,h}^{+}$ consists of the vectors $X = \left( \begin{smallmatrix}x_{2} & x_{1} \\ x_{3} & -x_{2} \end{smallmatrix}\right) \in  L+h$ with $Q(X) = m$ and $x_{3} > 0$, $z_{X}$ is the CM point associated to $X$, i.e., the unique root of $Q_{X}(z)$ in $\H$, and $\overline{\Gamma}_{X}$ is the stabilizer of $X$ in $\overline{\Gamma} = \Gamma/\{\pm 1\}$.
	
	For a harmonic Maass form $G \in H_{2k+2}(\Gamma)$ and $m \in \Q_{ < 0}, h \in L'/L,$ with $m \in \Z + Q(h)$ we define the trace function
	\begin{align*}
	\tr(G,m,h) &= \sum_{X \in \Gamma \backslash L_{m,h}}\int_{c(X)}^{\reg}G(z)Q_{X}^{k}(z)dz,
	\end{align*}
	with the regularized cycle integral defined in Section \ref{section regularized cycle integrals}.
	
	For a harmonic Maass form $F \in H_{-2k}^{+}(\Gamma)$ and $m \in \Q_{< 0}$ with $|m|/N$ being a square, say $m = -Nd^{2}$ for some $d \in \Q$, we define the complementary trace function
	\begin{align*}
	\tr^{c}(F,-Nd^{2},h) &= \sum_{X \in \Gamma \backslash L_{-Nd^{2},h}}\sum_{\substack{n \in \frac{1}{\alpha_{\ell_{X}}}\Z \\ n < 0}}a_{F,\ell_{X}}^{+}(n)(4\pi n)^{k}e^{2\pi i r_{\ell_{X}}n},
	\end{align*}
	where $\ell_{X}$ is the cusp associated to $X$, $r_{\ell_{X}}$ is the real part of the geodesic $c(X)$, and $a_{F,\ell_{X}}^{+}(w)$ are the coefficients of the holomorphic part of $F|_{-2k}\sigma_{\ell_{X}}$, see Section \ref{section regularized cycle integrals}.

	We require the special functions 
	\[
	\beta_{3/2+k}(v) = \int_{1}^{\infty}e^{-vt}t^{-3/2-k}dt, \qquad \beta^{c}_{3/2+k}(v) = \int_{0}^{1}e^{-vt}t^{-3/2-k}dt, \qquad (v > 0),
	\]
	where $\beta_{3/2+k}^{c}(v)$ is defined by the integral above only for $k \leq -1$, and can be analytically continued in $k$ similarly as the Gamma function. 
	
	Further, for a function $f(h)$ depending on $h \in L'/L$ we let $f(\pm h) = f(h) + (-1)^{k+1}f(-h)$.

	\begin{thm}\label{theorem fourier expansion shintani}
	Let $k \in \Z$ with $k > 0$ and let $G \in H_{2k+2}(\Gamma)$. If the constant coefficients $a_{\ell}^{+}(0)$ and $a_{\ell}^{-}(0)$ of $G$ at all cusps vanish, then the $h$-th component of $\ISh(G,\tau)$ is given by
	\begin{align*}
	\ISh_{h}(G,\tau)=& (-1)^{k}\sqrt{N}\sum_{m > 0}\tr(G,-m,h)q^{m} \\
	&-\sum_{m < 0}\pi\left( \frac{-\sqrt{N}}{4\pi\sqrt{|m|}}\right)^{k+1}\overline{\tr^{+}(R_{-2k}^{k}\xi_{2k+2}G,-m,\pm h)}v^{-k-1/2}\beta_{3/2+k}(4\pi |m|v)q^{m} \\
	&-\sum_{d > 0}\pi\left(\frac{i}{4\pi d}\right)^{k+1}\overline{\tr^{c}(\xi_{2k+2} G,-Nd^{2},\pm h)}v^{-k-1/2}\beta^{c}_{3/2+k}(-4\pi Nd^{2}v)q^{Nd^{2}}.
	\end{align*}
	If the constant coefficients of $G$ do not vanish, we have the additional contributions	
	\begin{align*}
	&+\frac{N^{k+1/2}}{2}\sum_{\substack{\ell \in \Gamma \backslash \Iso(V) \\ \ell \cap (L + h) \neq \emptyset}}\alpha_{\ell}\beta_{\ell}^{k}a_{\ell}^{+}(0)\big((-1)^{k+1}\zeta(s,\kappa_{\ell}/\beta_{\ell})+\zeta(s,1-\kappa_{\ell}/\beta_{\ell})\big)\big|_{s = -k} \\
	&+ \frac{k!}{2\pi^{k+1}}v^{-k-1/2}\sum_{\substack{\ell \in \Gamma \backslash \Iso(V) \\ \ell \cap (L + h) \neq \emptyset}}\frac{\alpha_{\ell}}{\beta_{\ell}^{k+1}}a_{\ell}^{-}(0)\big((-1)^{k+1}\zeta(s,\kappa_{\ell}/\beta_{\ell})+\zeta(s,1-\kappa_{\ell}/\beta_{\ell})\big)\big|_{s = k+1},
	\end{align*}
	where $\zeta(s,\rho) = \sum_{n = 0, n+ \rho \neq 0}^{\infty}(n+\rho)^{-s}$ is the Hurwitz zeta function and $\kappa_{\ell} \in \Q$ with $0 \leq \kappa_{\ell} < \beta_{\ell}$ is defined by $\sigma_{\ell}^{-1}h_{\ell} = \left(\begin{smallmatrix}0 & \kappa_{\ell} \\ 0 & 0 \end{smallmatrix}\right)$ for some $h_{\ell} \in \ell \cap (L+h)$. 
	
	For $k = 0$ the $h$-component is given by the same formula as above, but with the additional non-holomorphic terms
	\begin{align*}
	\frac{1}{\sqrt{N}}\sum_{\ell \in \Gamma \backslash \Iso(V)}\varepsilon_{\ell}a_{\ell}^{-}(0)\sum_{m > 0}b_{\ell,1}(m,h)\calF\left(4\pi mv\right)q^{m},
	\end{align*}
	where
	\begin{align*}
	\calF(w) &= \frac{\sqrt{\pi}}{2}w^{-1/2}e^{w}\erfc\left(\sqrt{w}\right)-\sqrt{\pi}\int_{0}^{\sqrt{w}}e^{t^{2}}\erfc(t)dt + \frac{1}{2}\log(w)+\log(2)-\frac{1}{2}\Gamma'(1).
	\end{align*}
	\end{thm}
	
	\begin{rem}
		\begin{enumerate}
			\item The Hurwitz zeta function $\zeta(s,\rho)$ has an analytic continuation to all $s \in \C$ with a simple pole at $s = 1$ with residue $1$, so for $k = 0$ the poles of $\zeta(s,\kappa_{\ell}/\beta_{\ell})$ and $\zeta(s,1-\kappa_{\ell}/\beta_{\ell})$ at $s = 1$ cancel out, and for $k > 0$ we can just plug in $s = -k$ and $s = k+1$.
			\item Using 
			\begin{align*}
			\xi_{3/2+k}\big(v^{-k-1/2}\beta_{3/2+k}(4\pi |m|v)q^{m}\big) &= -q^{-m}, \\
			\xi_{3/2+k}\big(v^{-k-1/2}\beta_{3/2+k}^{c}(4\pi Nd^{2}v)q^{Nd^{2}}\big) &= q^{-Nd^{2}},
			\end{align*}
			and the Fourier expansion of the Millson lift given in \cite{alfesschwagen}, Theorem 5.1, we obtain another proof of Proposition \ref{proposition xi diagram}.
			\item Note that $4\Delta_{3/2}\mathcal{F}(4\pi m v)e(m\tau) = -e(m\tau)$ and that all other terms in the Fourier expansion of $\ISh(G,\tau)$ are harmonic, which completes the proof of Proposition \ref{proposition Delta}.
			\item The Fourier expansion can be twisted using the method developed in \cite{ae}. We applied this method to the Millson theta lift in \cite{alfesschwagen}, Section 6. The twisting of the Shintani theta lift is very similar, so we do not give the details here. In order to obtain the scalar valued results in the introduction, we use that the space of vector valued harmonic Maass forms of weight $3/2+k$ for the Weil representation of the signature $(1,2)$ lattice $(\Z,-x^{2}) \oplus U$ (with $U$ a hyperbolic plane) is isomorphic to the space of scalar valued harmonic Maass forms of weight $3/2+k$ for $\Gamma_{0}(4)$ satisfying the Kohnen plus space condition by the map $f_{0}(\tau)\e_{0} + f_{1}(\tau)\e_{1} \mapsto f_{0}(4\tau) + f_{1}(4\tau)$, see \cite{ez}, Theorem~5.1.
		\end{enumerate}
	\end{rem}

	For brevity, we define
	\[
	\varphi_{Sh}^{0}(X,\tau,z) = v^{1/2}y^{-2k-2}Q_{X}^{k+1}(\bar{z})e^{-\pi v p_{z}^{2}(X)}.
	\]
	Throughout, we let $G \in H_{2k+2}(\Gamma)$ and we write $G|_{2k+2}\sigma_{\ell} = G_{\ell}$. The Shintani lift of $G$ has a Fourier expansion of the shape
	\[
	\ISh(G,\tau) = \sum_{h \in L'/L}\sum_{m \in \Q}C(m,h,v)e(m\tau)\e_{h},
	\]
	where for $k > 0$ the coefficients are given by
	\begin{align*}
	C(m,h,v) &= \sum_{\ell \in \Gamma \backslash \Iso(V)}\lim_{T \to \infty}\bigg[\int_{\calF_{T}^{\alpha_{\ell}}}G_{\ell}(z)\sum_{X \in L_{-m,h}}\overline{\varphi_{Sh}^{0}(\sigma_{\ell}^{-1}X,\tau,z)}y^{2k+2}d\mu(z) \\
	&\qquad \qquad \qquad \qquad\qquad -\frac{1}{\sqrt{N}}\varepsilon_{\ell}b_{\ell,k+1}(m,h)\left(a_{\ell}^{+}(0)\frac{T^{k+1}}{k+1}-a_{\ell}^{-}(0)\frac{T^{-k}}{k} \right)\bigg],
	\end{align*}
	and for $k = 0$ one has to replace $\frac{T^{-k}}{k}$ by $-\log(T)$. Here $b_{\ell,k+1}(m,h)$ denotes the Fourier coefficient at $q^{m}$ of the $h$-component of the theta function $\overline{\Theta_{\ell,k+1}(\tau)}$.

	\subsection{A $\Delta_{2k+2}$-preimage for the Shintani schwarz function}

We construct a preimage for the Shintani Schwartz function $\varphi^{0}_{Sh}(X,\tau,z)$ under $\Delta_{2k+2}$, which, together with Stokes' theorem, reduces the computation of the Fourier coefficients of the Shintani lift to the computation of certain boundary integrals. For $X \in V$ and $\tau,z \in \H$ with $Q_{X}(z) \neq 0$ we define the function
\begin{align}\label{eq eta main formula}
\eta(X,\tau,z) = -\frac{ 2^{k-1}N^{k+1}}{Q_{X}^{k+1}(z)}\int_{|p_{z}(X)|}^{\infty}\left(\frac{1}{2}t^{2}-(X,X)\right)^{k}\erfc\left(\sqrt{\pi v}t\right)dt,
\end{align}
where $\erfc(x) = \frac{2}{\sqrt{\pi}}\int_{x}^{\infty}e^{-w^{2}}dw$ is the complementary error function.

\begin{rem}
	A simple change of variables yields the alternative formula
	\begin{align}\label{eq eta alternative formula}
	\eta(X,\tau,z)	&= -\frac{2^{k}N^{k+1}}{Q_{X}^{k+1}(z)}\sqrt{v}\int_{1}^{\infty}\int_{R(X,z)}^{\infty}t^{k}e^{-2\pi v (t+(X,X)) w^{2}}dt \, dw.
	\end{align}
	\end{rem}

 For the computation of the Fourier coefficients of the Shintani lift we will need the following differential equations and growth estimates for $\eta(X,\tau,z)$.

\begin{lem}
	\begin{enumerate}
		\item For $Q_{X}(z) \neq 0$ we have
	\begin{align*}
	\xi_{2k+2} \eta(X,\tau,z) = \frac{\sqrt{N}}{2}Q_{X}^{k}(z) \sgn(p_{z}(X))\erfc\left(\sqrt{\pi v}|p_{z}(X)|\right)
	\end{align*}
	and
	\[
	\Delta_{2k+2} \eta(X,\tau,z) = \varphi_{Sh}^{0}(X,\tau,z).
	\]
		\item For $Q(X) > 0$ the functions $\eta(X,\tau,z)$ and $\xi_{2k+2}\eta(X,\tau,z)$ are square exponentially decreasing as $y \to 0$ and $y \to \infty$, uniformly in $x$, and as $x \to \infty$, uniformly in $y$.
		\item For $Q(X) < 0$ such that the stabilizer $\overline{\Gamma}_{X}$ is infinite cyclic, the functions $\eta(X,\tau,z)$ and $\xi_{2k+2}\eta(X,\tau,z)$ are square exponentially decreasing at the boundary of $\Gamma_{X} \backslash \H$.
	\end{enumerate}
\end{lem} 

\begin{proof}
	\begin{enumerate}
	\item This follows by a straightforward calculation using the rules \eqref{eq:diffpQR}.
	\item For $X = \left(\begin{smallmatrix}x_{2} & x_{1} \\ x_{3} & -x_{2}\end{smallmatrix} \right)$ with $Q(X) > 0$ we have $x_{3} \neq 0$ and 
	\begin{align}\label{eq pz2 alternative form}
	-\pi p_{z}^{2}(X) = -\pi \left( \frac{N(x_{3}x-x_{2})^{2}+q(X)}{\sqrt{N}x_{3}y} + \sqrt{N}x_{3}y\right)^{2}.
	\end{align}
	This implies the claimed decay of $\eta(X,\tau,z)$ and $\xi_{2k+2}\eta(X,\tau,z)$.
	\item Let $Q(X) = m < 0$ such that $\overline{\Gamma}_{X}$ is infinite cyclic. By conjugation with a matrix in $\SL_{2}(\R)$ we can assume that $X$ has the form $X = \sqrt{|m|/N}\left(\begin{smallmatrix}1 & 0 \\ 0 & - 1 \end{smallmatrix}\right)$.
	Then $\overline{\Gamma}_{X}$ is generated by a matrix $\left(\begin{smallmatrix}\varepsilon & 0 \\ 0 & \varepsilon^{-1}\end{smallmatrix}\right)$
	for some $\varepsilon > 1$, and a fundamental domain for $\Gamma_{X}\backslash \H$ is given by the set $\{z \in \H: 1 \leq |z| \leq \varepsilon^{2}\}$. The two semi-arcs $|z| = 1$ and $|z| = \varepsilon^{2}$ are identified via $\Gamma_{X}$, so letting $z$ go to the boundary of $\Gamma_{X}\backslash \H$ means letting $y$ go to $0$ with $x$ being bounded away from $0$. Explicitly, we have $p_{z}(X) = 2\sqrt{|m|}\,\frac{x}{y}$. The claimed decay of $\eta(X,\tau,z)$ and $\xi_{2k+2}\eta(X,\tau,z)$ at the boundary of $\Gamma_{X}\backslash \H$ is now easy to check.
	\end{enumerate}
\end{proof}

		\begin{rem}
			For $Q_{X}(z) \neq 0$ the function
			\begin{align*}
			\mu(X,\tau,z) 
		&= -\frac{Q_{X}^{k}(z)\sgn(p_{z}(X))}{2(2\pi)^{k+1}}\int_{|p_{z}(X)|}^{\infty}\frac{\Gamma\left(k+1,2\pi v\left(\frac{1}{2}t^{2}-(X,X)\right)\right)}{\left(\frac{1}{2}t^{2}-(X,X)\right)^{k+1}}dt
			\end{align*}
			satisfies
			\[
			\xi_{-2k,z}\mu(X,\tau,z) = \frac{N^{k+1/2}}{2\pi^{k+1} Q_{X}^{k+1}(z)}\Gamma\left(k+1,2\pi vR(X,z)\right) 
			\]
			and
			\[
			\Delta_{-2k,z}\mu(X,\tau,z) = \overline{\psi_{M,k}^{0}(X,\tau,z)},
			\]
			where we wrote $\psi_{M,k}(X,\tau,z) = \psi_{M,k}^{0}(X,\tau,z)e(Q(X)\tau)$. It can be used to compute the Fourier expansion of the Millson lift (see \cite{alfesschwagen}), in a similar way as we will use $\eta(X,\tau,z)$ below to compute the expansion of the Shintani lift. 
		\end{rem}

		\subsection{Coefficients of negative index}
	
	Let $m \in \Q$ with $m < 0$ and let $h \in L'/L$ with $m \in \Z - Q(h)$. By Proposition \ref{proposition convergence shintani lift} we have
	\[
	C(m,h,v) = \lim_{T \to \infty}\int_{\calF(\Gamma)_{T}}G(z)\sum_{X \in L_{-m,h}}\overline{\varphi_{Sh}^{0}(X,\tau,z)}y^{2k+2}d\mu(z).
	\]
	We use the unfolding argument, formally at first, to obtain
	\[
	C(m,h,v) = \sum_{X \in \Gamma \backslash L_{-m,h}}\frac{1}{|\overline{\Gamma}_{X}|}\int_{\H}G(z)\overline{\varphi_{Sh}^{0}(X,\tau,z)}y^{2k+2}d\mu(z).
	\]
	For $X = \left(\begin{smallmatrix}x_{2} & x_{1} \\ x_{3} & -x_{2}\end{smallmatrix} \right)$ with $Q(X) > 0$ we have $x_{3} \neq 0$, and using \eqref{eq pz2 alternative form} we see that the function $\varphi_{Sh}^{0}(X,\tau,z)$ is of square exponential decay as $y \to 0$ and $y \to \infty$, uniformly in $x$, and also as $x \to \pm\infty$, uniformly in $y$, so the last integral exists and the unfolding is justified.
	
	We split the sum over $X = \left( \begin{smallmatrix}x_{2} & x_{1} \\ x_{3} & -x_{2} \end{smallmatrix}\right) \in \Gamma \backslash L_{-m,h}$ according to the sign of $x_{3}$. In the sum corresponding to $x_{3}< 0$ we can replace $X$ by $-X$, giving a factor $(-1)^{k+1}$ and replacing $h$ by $-h$, so it suffices to compute the integral in $C(m,h,v)$ for fixed $X \in L_{-m,h}$ with $x_{3} > 0$. Then the CM point corresponding to $X$ is explicitly given by
	\[
	z_{X} = \frac{x_{2}+i\sqrt{Q(X)/N}}{x_{3}}.
	\]
	For $\varepsilon > 0$ we let $B_{\varepsilon}(z_{X}) = \{z \in \H: R(X,z) \leq \varepsilon\}$	be the $\varepsilon$-ball around $z_{X}$. Using the $\Delta_{2k+2}$-preimage $\eta(X,\tau,z)$ of $\varphi_{Sh}^{0}(X,\tau,z)$ and Stokes' theorem we obtain
	\begin{align*}
	\int_{\H}G(z)\overline{\varphi_{Sh}^{0}(X,\tau,z)}y^{2k+2}d\mu(z) &= \lim_{\varepsilon \to 0}\int_{\partial B_{\varepsilon}(z_{X})}G(z)\xi_{2k+2}\eta(X,\tau,z)dz\\
	&\quad - \lim_{\varepsilon \to 0}\overline{\int_{\partial B_{\varepsilon}(z_{X})}\xi_{2k+2}G(z)\eta(X,\tau,z)dz}.
	\end{align*}
	The first integral vanishes as $\varepsilon \to 0$ since $\xi_{2k+2}\eta(X,\tau,z)$ is smooth at $z_{X}$. Further, for $z \in \partial B_{\varepsilon}(z_{X})$ we have $|p_{X}(z)| = \sqrt{2\varepsilon +4Q(X)}$, so we can further write
	\begin{align*}
	-\int_{\partial B_{\varepsilon}(z_{X})}\xi_{2k+2}G(z)\eta(X,\tau,z)dz &= 2^{k-1}N^{k+1}\int_{\partial B_{\varepsilon}(z_{X})}\xi_{2k+2}G(z)\frac{dz}{Q_{X}^{k+1}(z)} \\
	&\qquad \times \int_{\sqrt{2\varepsilon +4Q(X)}}^{\infty}\left(\frac{1}{2}t^{2}-(X,X)\right)^{k}\erfc\left(\sqrt{\pi v}t\right)dt.
	\end{align*}
	The Laurent expansion of the meromorphic function $Q_{X}^{-k-1}(z)$ around $z_{X}$ is given by
	\begin{align*}
	\frac{1}{Q_{X}^{k+1}(z)} &= 
	 \frac{1}{(Nx_{3})^{k+1}}\frac{1}{(z-z_{X})^{k+1}}\frac{1}{(z_{X}-\bar{z}_{X})^{k+1}}\sum_{r=0}^{\infty}\binom{r+k}{k}\left(\frac{z-z_{X}}{\bar{z}_{X}-z_{X}} \right)^{r}
	\end{align*}
	for $|z-z_{X}| < |\bar{z}_{X}-z_{X}|$. 

	Since $\xi_{2k+2}G$ is holomorphic on $\H$ we obtain by the residue theorem for $\varepsilon$ small enough
	\begin{align*}
	\int_{\partial B_{\varepsilon}(z_{X})}\xi_{2k+2}G(z)\frac{dz}{Q_{X}^{k+1}(z)} 
	&= \frac{2\pi i}{(Nx_{3}(z_{X}-\bar{z}_{X}))^{k+1}}\sum_{r=0}^{k}\binom{2k-r}{k}\frac{(\xi_{2k+2}G)^{(r)}(z_{X})}{(\bar{z}_{X}-z_{X})^{k-r} r!}.
	\end{align*}
	For a smooth function $F$ transforming like a modular form of weight $-2k$ we have the formula
\[
	R_{-2k}^{k}F(z) = (2i)^{k}k!\sum_{r = 0}^{k}\binom{2k-r}{k}\frac{F^{(r)}(z)}{(-2iy)^{k-r}r!},
	\]
	which can be proven by induction. Using $z_{X} -\bar{z}_{X} = 2i\Im(z_{X}) = i\sqrt{4NQ(X)}/Nx_{3}$, we therefore obtain
	\begin{align*}
	\int_{\partial B_{\varepsilon}(z_{X})}\xi_{2k+2}G(z)\frac{dz}{Q_{X}^{k+1}(z)} =  \frac{2\pi}{(-2)^{k}k!\sqrt{4NQ(X)}^{k+1}}(R_{-2k}^{k}\xi_{2k+2}G)(z_{X}).
	\end{align*}
	Finally, a short calculation using the definition of the complementary error function gives
	\[
	\int_{\sqrt{4Q(X)}}^{\infty}\left(\frac{1}{2}t^{2}-(X,X)\right)^{k}\erfc\left(\sqrt{\pi v}t\right)dt = \frac{2^{k}k!}{\sqrt{\pi}}(4\pi v)^{-k-1/2}\beta_{3/2-k}(4\pi Q(X)v).
	\]
	Taking everything together and writing $Q(X) = |m|$, we arrive at the formula
	\begin{align*}
	C(m,h,v) &= (-1)^{k}\pi \left(\frac{\sqrt{N}}{4\pi\sqrt{m}} \right)^{k+1}v^{-k-1/2}\beta_{3/2-k}(4\pi|m|v) \\
	&\quad\times \left(\overline{\tr^{+}(R_{-2k}^{k}\xi_{2k+2}G,-m,h)} + (-1)^{k+1}\overline{\tr^{+}(R_{-2k}^{k}\xi_{2k+2}G,-m,-h)}\right),
	\end{align*}
	which completes the computation of the coefficients of negative index.
	
	\subsection{Coefficients of index $0$} 
	
	Let $h \in L'/L$ with $Q(h) \equiv 0 \mod \Z$. By Proposition \ref{proposition convergence shintani lift} the coefficient of index $(0,h)$ is given by
	\begin{align*}
	C(0,h,v) &= \lim_{T \to \infty}\bigg[\int_{\calF(\Gamma)_{T}}G(z)\sum_{X \in L_{0,h}}\overline{\varphi_{Sh}^{0}(X,\tau,z)}y^{2k+2}d\mu(z) \\
	&\qquad \qquad  -\frac{1}{\sqrt{N}}\sum_{\ell \in \Gamma \backslash \Iso(V)}\varepsilon_{\ell}b_{\ell,k+1}(0,h)\left(a_{\ell}^{+}(0)\frac{T^{k+1}}{k+1}-a_{\ell}^{-}(0)\frac{T^{-k}}{k} \right)\bigg],
	\end{align*}
	with $\frac{T^{-k}}{k}$ replaced by $-\log(T)$ if $k = 0$. Note that $\varphi_{Sh}^{0}(0,\tau,z) = 0$, so we can omit the summand for $X = 0$. Since $\eta(X,\tau,z)$ and $\xi_{2k+2}\eta(X,\tau,z)$ are continuous for all $z \in \H$ if $X \in V \setminus \{0\}$ with $Q(X) = 0$, we obtain by Stokes' theorem
	\begin{align}
		&\int_{\calF(\Gamma)_{T}}G(z)\sum_{X \in L_{0,h}\setminus \{0\}}\overline{\varphi_{Sh}^{0}(X,\tau,z)}y^{2k+2}d\mu(z)\\
		&= - \sum_{\ell \in \Gamma \backslash \Iso(V)}\int_{iT}^{iT+\alpha_{\ell}}G_{\ell}(z)\sum_{X \in L_{0,h}\setminus \{0\}}\xi_{2k+2}\eta(\sigma_{\ell}^{-1}X,\tau,z)dz \label{eq zero coefficient boundary 1} \\
		&\quad + \sum_{\ell \in \Gamma \backslash \Iso(V)}\overline{\int_{iT}^{iT+\alpha_{\ell}}\xi_{2k+2}G_{\ell}(z)\sum_{X \in L_{0,h}\setminus \{0\}}\eta(\sigma_{\ell}^{-1}X,\tau,z)dz} \label{eq zero coefficient boundary 2}.
	\end{align}
	We only compute the first boundary integral \eqref{eq zero coefficient boundary 1}. The second one can be treated in a similar way. We parametrize the set $L_{0,h}\backslash \{0\}$ by the elements $nX_{\ell'} + h_{\ell'}$, where $\ell' \in \Iso(V)$ runs through all isotropic lines with $\ell' \cap (L + h) \neq \emptyset$, $X_{\ell'} \in \ell' \cap L$ is a primitive and positively oriented generator of $\ell'$, and $n \in \Z$ runs such that $n X_{\ell'} + h_{\ell'} \neq 0$. Thus we can write \eqref{eq zero coefficient boundary 1} as
	\begin{align*}
	-\sum_{\substack{\ell' \in \Gamma \backslash \Iso(V) \\ \ell' \cap (L + h) \neq \emptyset}}\int_{iT}^{iT+\alpha_{\ell}}G_{\ell}(z)\sum_{\gamma \in \Gamma/\Gamma_{\ell'}}\sum_{\substack{n \in \Z \\ nX_{\ell'}+h_{\ell'} \neq 0}}\xi_{2k+2}\eta(\sigma_{\ell}^{-1}(nX_{\ell'} + h_{\ell'}),\tau,z)|_{2k+2}\gamma \, dz.
	\end{align*}
	Explicitly, we have $\sigma_{\ell'}^{-1}(nX_{\ell'} + h_{\ell'}) = \left(\begin{smallmatrix}0 & n\beta_{\ell'} + \kappa_{\ell'} \\ 0 & 0 \end{smallmatrix}\right)$ for some $\kappa_{\ell'} \in \Q$ with $0 \leq \kappa_{\ell'} < \beta_{\ell'}$, and
	\begin{align*}
	&\xi_{2k+2}\eta(\sigma_{\ell}^{-1}(nX_{\ell'} + h_{\ell'}),\tau,z)\\
	&= \left(\frac{\sqrt{N}}{2}(-N(n\beta_{\ell'}+\kappa_{\ell'}))^{k}\sgn(n\beta_{\ell'}+\kappa_{\ell'})\erfc\left(\sqrt{N\pi v}\frac{|n\beta_{\ell'} + \kappa_{\ell'}|}{y}\right) \right)\bigg|_{2k+2}\sigma_{\ell'}^{-1}\sigma_{\ell}.
	\end{align*}
	We see that for $\ell \neq \ell' \mod \Gamma$ the sum
	\[
	\sum_{\gamma \in \Gamma/\Gamma_{\ell'}}\sum_{\substack{n \in \Z \\ n\beta_{\ell'}+\kappa_{\ell'}\neq 0}}\xi_{2k+2}\eta(\sigma_{\ell}^{-1}(nX_{\ell'} + h_{\ell'}),\tau,z)|_{2k+2}\gamma
	\]
	is square exponentially decreasing as $y \to \infty$, and for $\ell' = \ell \mod \Gamma$ the sum over $\gamma \neq 1$ is square exponentially decreasing as $y \to \infty$ as well, so in the limit $T \to \infty$ only the terms with $\ell' = \ell$ and $\gamma = 1$ survive. Thus we find that \eqref{eq zero coefficient boundary 1} vanishes in the limit if $\ell \cap (L + h) = \emptyset$, and otherwise behaves like
	\begin{align*}
	\frac{(-N)^{k+1}}{T}\sqrt{v}\alpha_{\ell}\left(a_{\ell}^{+}(0)+\frac{a_{\ell}^{-}(0)}{T^{2k+1}}\right)\int_{1}^{\infty}\sum_{\substack{n \in \Z \\ n\beta_{\ell} + \kappa_{\ell} \neq 0}}(n\beta_{\ell}+\kappa_{\ell})^{k+1}e^{-N\pi v \frac{(n\beta_{\ell}+\kappa_{\ell})^{2}w^{2}}{T^{2}}}dw + O(e^{-\varepsilon T^{2}})
	\end{align*}
	as $T \to \infty$. We assume from now on that $\ell \cap (L+h) \neq \emptyset$. For $s \in \C$ we let
	\begin{align}
	\Omega(s) &= \frac{(-N)^{k+1}}{T}\sqrt{v}\alpha_{\ell}\int_{1}^{\infty}\sum_{\substack{n \in \Z \\ n\beta_{\ell} + \kappa_{\ell} \neq 0}}(n\beta_{\ell}+\kappa_{\ell})^{k+1}e^{-N\pi v \frac{(n\beta_{\ell}+\kappa_{\ell})^{2}w^{2}}{T^{2}}}w^{s}dw, \label{eq zero coefficient omega}
	\end{align}
	and we split the integral in $\Omega(s)$ as $\int_{1}^{\infty} = \int_{0}^{\infty}-\int_{0}^{1}$. In the part corresponding to $\int_{0}^{\infty}$ we can interchange the sum and the integral if $\Re(s) > k+1$, so this part equals
	\begin{align*}
	&\alpha_{\ell}\beta_{\ell}^{k-s}\frac{N^{k+1/2-s/2}\pi^{-(s+1)/2}}{2T^{2s}v^{s/2}}\Gamma\left(\frac{s+1}{2}\right)\left( (-1)^{k+1}\zeta(s-k,\kappa_{\ell}/\beta_{\ell}) + \zeta(s-k,1-\kappa_{\ell}/\beta_{\ell})\right),
	\end{align*}
	where $\zeta(s,\rho) = \sum_{n = 0, n + \rho \neq 0}^{\infty}(n+\rho)^{-s}$ is the Hurwitz zeta function. It has an analytic continuation to $\C$ with a simple pole at $s = 1$, so the above expression is holomorphic at $s = 0$. 
	
	In the part corresponding to $\int_{0}^{1}$ we apply Poisson summation to obtain
	\begin{align*}
	&-\frac{(-N)^{k+1}}{T}\sqrt{v}\alpha_{\ell}\int_{0}^{1}\sum_{\substack{n \in \Z \\ n\beta_{\ell} + \kappa_{\ell} \neq 0}}(n\beta_{\ell}+\kappa_{\ell})^{k+1}e^{-N\pi v \frac{(n\beta_{\ell}+\kappa_{\ell})^{2}w^{2}}{T^{2}}}w^{s}dw \\
	&= -\frac{\varepsilon_{\ell}}{\sqrt{N}}\cdot\frac{(-\sqrt{N}i)^{k+1}}{(4\pi v)^{(k+1)/2}}T^{k+1}\sum_{m \in \Z}e^{-2\pi i m \kappa_{\ell}/\beta_{\ell}}\int_{0}^{1}H_{k+1}\left(\frac{\sqrt{\pi}\, T m}{\sqrt{Nv}\beta_{\ell}w}\right) e^{-\frac{\pi T^{2}m^{2}}{Nv\beta_{\ell}^{2}w^{2}}}w^{s-k-2}dw.
	\end{align*}
	The summands for $m \neq 0$ are holomorphic for all $s \in \C$ and square exponentially decreasing as $T\to \infty$. By computing the integral in the summand for $m = 0$ we see that it has an analytic continuation to $s = 0$, and we find that  the above expression behaves like
	\begin{align*}
	\frac{\varepsilon_{\ell}}{\sqrt{N}}\cdot\frac{(-\sqrt{N}i)^{k+1}}{(4\pi v)^{(k+1)/2}}\cdot\frac{T^{k+1}}{k+1}H_{k+1}(0) + O(e^{-\varepsilon' T^{2}})
	\end{align*}
	as $T \to \infty$. Note that $b_{\ell,k+1}(0,h) = \frac{(-\sqrt{N}i)^{k+1}}{(4\pi v)^{(k+1)/2}}H_{k+1}(0)$ by \eqref{eq zero coefficients unary theta functions}.
	
	Together, we obtain
	\begin{align*}
	\Omega(0) &= \frac{N^{k+1/2}}{2}\alpha_{\ell}\beta_{\ell}^{k}\big((-1)^{k+1}\zeta(-k,\kappa_{\ell}/\beta_{\ell})+\zeta(-k,1-\kappa_{\ell}/\beta_{\ell})\big)\\
	&\quad  + \frac{1}{\sqrt{N}}\varepsilon_{\ell}b_{\ell,k+1}(0,h)\frac{T^{k+1}}{k+1} + O(e^{-\varepsilon'' T^{2}})
	\end{align*}
	as $T \to \infty$. This gives the contribution to the Fourier expansion as stated in the theorem.

\subsection{Coefficients of positive index}

	Let $m\in\Q$ with $m > 0$ and let $h\in L'/L$ with $h\in \Z-Q(h)$. We need to investigate
	\begin{align}\label{eq coefficients of positive index}
	C(m,h,v) &= \lim_{T \to \infty}\bigg[\sum_{X\in \Gamma\backslash L_{-m,h}} \int_{\calF(\Gamma)_{T}}G(z)\sum_{\gamma \in \overline{\Gamma}_{X}\backslash \overline{\Gamma}}\overline{\varphi_{Sh}^{0}(\gamma X,\tau,z)}y^{2k+2}d\mu(z) \\
	&\qquad \qquad \qquad -\frac{1}{\sqrt{N}}\sum_{\ell \in \Gamma \backslash \Iso(V)}\varepsilon_{\ell}b_{\ell,k+1}(m,h)\left(a_{\ell}^{+}(0)\frac{T^{k+1}}{k+1}-a_{\ell}^{-}(0)\frac{T^{-k}}{k} \right)\bigg],
	\end{align}
	with $\frac{T^{-k}}{k}$ replaced by $-\log(T)$ if $k = 0$. Note that $b_{\ell,k+1}(m,h) = 0$ if $m/N$ is not a square. We compute the integral for fixed $X\in L_{-m,h}$.
	
	We truncate the geodesic $c_{X}$ by setting $c_{X}^{T} = c_{X} \cap \calF(\Gamma)_{T}$ and define an $\varepsilon$-neighbourhood of $c_{X}^{T}$ in $\calF(\Gamma)_{T}$ by $B_{\varepsilon}(c_{X}^{T})=\{z \in \calF(\Gamma)_{T}:\,|p_{z}(X)|<\varepsilon\}$. Using the $\Delta_{2k+2}$-preimage of the Shintani theta function and Stokes' theorem we then obtain 
	\begin{align}
		&\int_{\calF(\Gamma)_{T}}G(z)\sum_{\gamma \in \overline{\Gamma}_{X}\backslash \overline{\Gamma}}\overline{\varphi_{Sh}^{0}(\gamma X,\tau,z)}y^{2k+2}d\mu(z)\notag \\
		&=-\overline{\int_{\partial \calF(\Gamma)_{T}}\xi_{2k+2}G(z)\sum_{\gamma \in \overline{\Gamma}_{X}\backslash \overline{\Gamma}}\eta(\gamma X,\tau,z)dz}\label{eq:geodesic2}\\
		&\quad-\lim_{\varepsilon \to 0}\overline{\int_{\partial B_{\varepsilon}(c_{X}^{T})}\xi_{2k+2}G(z)\sum_{\gamma \in \overline{\Gamma}_X\backslash \overline{\Gamma}}\eta(\gamma X,\tau,z)dz}\label{eq:geodesic3}\\
		&\quad + \int_{\partial \calF(\Gamma)_{T}}G(z)\sum_{\gamma \in \overline{\Gamma}_{X}\backslash \overline{\Gamma}}\xi_{2k+2}\eta(\gamma X,\tau,z)dz\label{eq:geodesic4}\\
		&\quad+\lim_{\varepsilon \to 0}\int_{\partial B_{\varepsilon}(c_{X}^{T})}G(z)\sum_{\gamma \in \overline{\Gamma}_{X}\backslash \overline{\Gamma}}\xi_{2k+2}\eta(\gamma X,\tau,z)dz, \label{eq:geodesic5}
	\end{align}
	where we already used that $G$ is harmonic. The integral in \eqref{eq:geodesic3} vanishes since $\eta$ is continuous along $c_{X}$.
	If $m/N$ is not a square, we can unfold and obtain that the integrals in \eqref{eq:geodesic2} and \eqref{eq:geodesic4} vanish by the rapid decay of $\eta$ and $\xi_{2k+2}\eta$ at the boundary of $\Gamma_X\backslash \H$.
	If $m/N$ is a square, we cannot unfold and have to evaluate the integrals in \eqref{eq:geodesic2} and \eqref{eq:geodesic4}.

	We first compute the integral in \eqref{eq:geodesic5}, which yields a truncated cycle integral.

	\begin{prop}
		We have
		\begin{align*}
		\lim_{\varepsilon \to 0}\int_{\partial B_{\varepsilon}(c_{X}^{T})}G(z)\sum_{\gamma \in \overline{\Gamma}_{X}\backslash \overline{\Gamma}}\xi_{2k+2}\eta(\gamma X,\tau,z)dz = (-1)^{k`}\sqrt{N}\int_{c_{X}^{T}}G(z)Q_{X}^{k}(z)dz,
		\end{align*}
		where $c_{X}^{T} = c_{X} \cap \calF(\Gamma)_{T}$.
	\end{prop}

	\begin{proof}
	For $z \in \partial B_{\varepsilon}(c_{X}^{T})$ we have 
	\[
	\xi_{2k+2}\eta(X,\tau,z) = \frac{\sqrt{N}}{2}Q_{X}^{k}(z)\sgn(p_{z}(X))\erfc(\sqrt{\pi v}\varepsilon).
	\]
	The proof now proceeds in the same way as in \cite{bif}, Proposition 7.6.

	\end{proof}
	
	This already finishes the computation if $m/N$ is not a square, so we assume for the rest of this section that $m/N$ is a square. Then $\overline{\Gamma}_{X}$ is trivial and $c(X)$ is an infinite geodesic in the modular curve $M$. Unwinding the definition of the truncated surface $\calF(\Gamma)_{T}$, we can write the integrals in \eqref{eq:geodesic4} and \eqref{eq:geodesic2} as
	\begin{align}\label{eq unwinded boundary integral 1}
	\int_{\partial \calF(\Gamma)_{T}}G(z)\sum_{\gamma \in \overline{\Gamma}}\xi_{2k+2}\eta(\gamma X,\tau,z)dz 
	= -\sum_{\ell \in \Gamma \backslash \Iso(V)}\int_{iT}^{iT + \alpha_{\ell}}G_{\ell}(z)\sum_{\gamma \in \overline{\Gamma}}\xi_{2k+2}\eta(\sigma_{\ell}^{-1}\gamma X,\tau,z)dz,
	\end{align}

\begin{align}\label{eq unwinded boundary integral 2}
\int_{\partial \calF(\Gamma)_{T}}\xi_{2k+2}G(z)\sum_{\gamma \in \overline{\Gamma}}\eta(\gamma X,\tau,z)dz 
		= -\sum_{\ell \in \Gamma \backslash \Iso(V)}\int_{iT}^{iT + \alpha_{\ell_{X}}}\xi_{2k+2}G_{\ell}(z)\sum_{\gamma \in \overline{\Gamma}}\eta(\sigma_{\ell}^{-1}\gamma X,\tau,z)dz,
	\end{align}
	where $G_{\ell} = G|_{2k+2}\sigma_{\ell}$. It can be seen as in the proof of \cite{brfu06}, Lemma 5.2, that in the limit only the summands for $\ell = \ell_{X}, \gamma \in \overline{\Gamma}_{\ell_{X}},$ and $\ell = \ell_{-X},\gamma \in \overline{\Gamma}_{\ell_{-X}},$ survive. Replacing $X$ by $-X$ in the part corresponding to $\ell_{-X}$ gives a factor $(-1)^{k+1}$ and changes $h$ to $-h$. In particular, it suffices to compute the part corresponding to $\ell_{X}$. We determine the asymptotic behaviour of $\sum_{\gamma \in \overline{\Gamma}}\xi_{2k+2}\eta(\sigma_{\ell}^{-1}\gamma X,\tau,z)$ and $\sum_{\gamma \in \overline{\Gamma}}\eta(\sigma_{\ell}^{-1}\gamma X,\tau,z)$.
	
	\begin{prop}\label{prop:firstboundaryintegral}
		Suppose that $m/N$ is a square. Let $X \in L_{-m,h}$, let $\ell_{X}$ be the cusp associated to $X$, and let $r_{\ell_{X}}$ be the real part of the geodesic $c(X)$. For $r_{\ell_{X}} < \Re(\sigma_{\ell_{X}}z) < r_{\ell_{X}}+\alpha_{\ell_{X}}$ we have
		\begin{align*}
			&\left(2\sqrt{mN}\right)^{-k}\sqrt{N}^{-1}\sum_{\gamma \in \overline{\Gamma}_{\ell_{X}}}\xi_{2k+2}\eta(\sigma_{\ell_{X}}^{-1}\gamma X,\tau,z) \\
			&= \frac{\alpha_{\ell_{X}}^{k}}{k+1}B_{k+1}\left(\frac{z-r_{_{\ell_{X}}}}{\alpha_{\ell_{X}}}\right) - \frac{(iy)^{k+1}}{\alpha_{\ell_{X}}(k+1)}\frac{H_{k+1}\left(2\sqrt{\pi mv}\right) }{\left(4\sqrt{\pi mv}\right)^{k+1}}+
			\sum_{\substack{n \in \frac{1}{\alpha_{\ell_{X}}}\Z\setminus \{0\}}}g_{n}(y)e^{-2\pi in(x-r_{\ell_{X}})},
		\end{align*}
		where $B_{n}(x) = -n\zeta(1-n,x)$ is the $n$-th Bernoulli polynomial, $H_{n}(x) = (-1)^{n}e^{x^{2}}\frac{d^{n}}{dx^{n}}e^{-x^{2}}$ is the $n$-th Hermite polynomial, and $g_{n}(y) = O(e^{-Cy^{2}})$ as $y \to \infty$, for some constant $C > 0$.
	\end{prop}

	\begin{proof}
	Let $\ell = \ell_{X},\alpha = \alpha_{\ell_{X}}$ and $r = r_{\ell_{X}}$ for brevity. We have
	\[
	X' = \sigma_{\ell}^{-1}X = \sqrt{\frac{m}{N}}\begin{pmatrix}1 & -2r\\ 0 & -1 \end{pmatrix}, \qquad \overline{\Gamma}_{\ell_{X'}} = \sigma_{\ell}^{-1}\overline{\Gamma}_{\ell}\sigma_{\ell}=\left\{\begin{pmatrix}1 &\alpha n \\ 0 & 1\end{pmatrix}: n\in \Z\right\},
	\]
	so that $c_{X'} = \{z \in \H: \Re(z) = r\}$ is a vertical geodesic, and we can write
	\begin{align*}
	&\sum_{X \in \overline{\Gamma}_{\ell}}\xi_{2k+2}\eta(\sigma_{\ell}^{-1}\gamma X ,\tau,z) = \sum_{n \in \Z}\xi_{2k+2}\eta\left(\sqrt{\frac{m}{N}}\begin{pmatrix}-1 & -2(\alpha n -r) \\ 0 & 1\end{pmatrix},\tau,z\right)\\
	&\quad= -\left(2\sqrt{mN}\right)^{k+1}\frac{\sqrt{v}}{y}\sum_{n\in\Z}\left(z+\alpha n -r\right)^{k}(x+\alpha n -r)\int_{1}^{\infty}e^{-4\pi mv\frac{(x+\alpha n - r)^{2}}{y^{2}}t^{2}}dt.
	\end{align*}
	Replacing $z$ by $z+r$ we can assume $r = 0$. 
	For $s \in \C$ with $\Re(s) > -1-k$ we consider
	\begin{align}
	&-\frac{\sqrt{v}}{y}\sum_{n \in \Z}(x+\alpha n)(z+\alpha n)^{k}\int_{0}^{\infty}e^{-4\pi mv\frac{(x+\alpha n)^{2}}{y^{2}}t^{2}}t^{s}dt \label{line 1} \\
	&+\frac{\sqrt{v}}{y}\sum_{n \in \Z}(x+\alpha n)(z+\alpha n)^{k}\int_{0}^{1}e^{-4\pi m v\frac{(x+\alpha n)^{2}}{y^{2}}t^{2}}t^{s}dt. \label{line 2}
	\end{align}
	The first line can be computed as
	\begin{align*}
	-2^{-s-2}v^{-s/2}y^{s}(\pi|m|)^{-(s+1)/2}\Gamma\left( \frac{s+1}{2}\right)\alpha^{k-s}\sum_{n \in \Z}\frac{\sgn\left(\frac{x}{\alpha}+n\right)\left(\frac{z}{\alpha}+ n\right)^{k}}{\left|\frac{x}{\alpha}+n\right|^{s}}.
	\end{align*}
	Further, we write (assuming that $0 < x < \alpha$)
	\begin{align*}
	\sum_{n \in \Z}\frac{\sgn\left(\frac{x}{\alpha}+n\right)\left(\frac{z}{\alpha}+ n\right)^{k}}{\left|\frac{x}{\alpha}+n\right|^{s}} = \sum_{n \geq 0}\frac{\left(\frac{z}{\alpha}+ n\right)^{k}}{\left(\frac{x}{\alpha}+n\right)^{s}} + (-1)^{k+1}\sum_{n \geq 0}\frac{\left(1-\frac{z}{\alpha}+ n\right)^{k}}{\left(1-\frac{x}{\alpha}+n\right)^{s}}.
	\end{align*}
	Expanding the numerator, we find
	\begin{align*}
	\sum_{n \geq 0}\frac{\left(\frac{z}{\alpha}+ n\right)^{k}}{\left(\frac{x}{\alpha}+n\right)^{s}} 
	&= \sum_{j=0}^{k}\binom{k}{j}\left( \frac{z-x}{\alpha}\right)^{k-j}\sum_{n\geq 0}\frac{1}{\left(\frac{x}{\alpha}+n\right)^{s-j}} = \sum_{j=0}^{k}\binom{k}{j}\left( \frac{iy}{\alpha}\right)^{k-j}\zeta\left(s-j,\frac{x}{\alpha}\right).
	\end{align*}
	The Hurwitz zeta function has a meromorphic continuation to $\C$ which is holomorphic at all non-positive integers, so we can just plug in $s = 0$. Let $B_{n}(x)$ be the $n$-th Bernoulli polynomial. Using the relations $\zeta(-j,w) = -B_{j+1}(w)/(j+1)$ for $j \geq 0$ and $B_{n}(a+b) = \sum_{j = 0}^{n}\binom{n}{j}a^{n-j}B_{j}(b)$ we further obtain
	\begin{align*}
	\sum_{n \geq 0}\frac{\left(\frac{z}{\alpha}+ n\right)^{k}}{\left(\frac{x}{\alpha}+n\right)^{s}}\bigg|_{s = 0}&= -\frac{1}{k+1}\sum_{j=0}^{k}\binom{k+1}{j+1}\left( \frac{iy}{\alpha}\right)^{k-j}B_{j+1}\left(\frac{x}{\alpha}\right) \\
	&= -\frac{1}{k+1}B_{k+1}\left(\frac{z}{\alpha}\right) + \frac{1}{k+1}\left(\frac{iy}{\alpha} \right)^{k+1}.
	\end{align*}
	Since $B_{j}(1-w) = (-1)^{j}B_{j}(w)$ we see that the second part with $1-\frac{z}{\alpha}$ gives exactly the same contribution, so we finally obtain for \eqref{line 1} evaluated at $s = 0$ the expression
	\begin{align*}
	\frac{\alpha^{k}}{2\sqrt{m}}\left(\frac{1}{k+1}B_{k+1}\left(\frac{z}{\alpha}\right) - \frac{1}{k+1}\left(\frac{iy}{\alpha} \right)^{k+1} \right).
	\end{align*}

	In the second line \eqref{line 2} above we replace $w = 1/t$ to get
	\begin{align*}
	\frac{\sqrt{v}}{y}\int_{1}^{\infty}\sum_{n \in \Z}(x+\alpha n)(z+\alpha n)^{k}e^{-4\pi m v\frac{(x+\alpha n)^{2}}{y^{2}w^{2}}}w^{-s-2}dw.
	\end{align*}
	For a function $g(t)$ on $\R$ we let $\hat{g}(\xi) = \int_{-\infty}^{\infty}g(t)e^{2\pi i t\xi}dt$ its Fourier transform. By Poisson summation, we can rewrite the sum as
	\begin{align*}
	\sum_{n \in \Z}(x+\alpha n)(z+\alpha n)^{k}e^{-4\pi mv\frac{(x+\alpha n)^{2}}{y^{2}w^{2}}}=\frac{1}{\alpha}\sum_{n \in \frac{1}{\alpha}\Z}e^{-2\pi inx}\int_{-\infty}^{\infty}t(t+iy)^{k}e^{-4\pi mv\frac{t^{2}}{y^{2}w^{2}}}e^{2\pi i nt}dt.
	\end{align*}
	We set $a = \frac{2\sqrt{\pi mv}}{yw}$ for the moment. Then the Fourier transform of
		$
		h_{k}(t) = t(t+iy)^{k}e^{-a^{2}t^{2}}
		$
		is given by
		\[
		\hat{h}_{k}(\xi) = \frac{\sqrt{\pi}}{a}\sum_{j=0}^{k}\binom{k}{j}(iy)^{k-j}\left(\frac{i}{2a}\right)^{j+1}H_{j+1}\left(\frac{\pi}{a}\xi\right)e^{-\frac{\pi^{2}\xi^{2}}{a^{2}}},
		\]
		which can be checked by a straightforward calculation using the rules
		$
		\widehat{e^{-a^{2}\xi^{2}}} = \frac{\sqrt{\pi}}{a}e^{-\pi^{2}\xi^{2}/a^{2}}$ and $ \widehat{t^{n}f(t)}(\xi) = (2\pi i)^{-n}\frac{d^{n}}{d\xi^{n}}\hat{f}(\xi)
		$.

	We find that \eqref{line 2} equals
	\begin{align*}
	&\frac{(iy)^{k+1}}{\alpha}\frac{1}{2\sqrt{m}}\sum_{n \in \frac{1}{\alpha}\Z}e^{-2\pi inx}\sum_{j=0}^{k}\binom{k}{j}\left(\frac{1}{4\sqrt{\pi mv}}\right)^{j+1} \int_{1}^{\infty}H_{j+1}\left(\frac{\sqrt{\pi}yw}{2\sqrt{mv}}n\right)e^{-\frac{\pi^{2}ny^{2}w^{2}}{4\pi mv}}w^{j-s}dw.
	\end{align*}
	For $n \neq 0$ the integral in the last line is holomorphic at $s = 0$ and square exponentially decreasing as $y \to \infty$, and for $n =0$ and $\Re(s) \gg 0$ it can be evaluated as $-(j+1-s)^{-1}H_{j+1}(0)$.

	We can now plug in $s = 0$. Using the identity $H_{n}(a+b) = \sum_{j=0}^{n}\binom{n}{j}(2a)^{n-j}H_{j}(b)$ we write
	\begin{align*}
	\sum_{j=0}^{k}\binom{k}{j}\left(\frac{1}{4\sqrt{\pi mv}}\right)^{j+1}\frac{H_{j+1}(0)}{j+1}
	&= \frac{1}{k+1}\frac{H_{k+1}\left( 2\sqrt{\pi mv}\right)}{\left(4\sqrt{\pi mv}\right)^{k+1}}-\frac{1}{k+1}.
	\end{align*}
	This finishes the proof.
	\end{proof}

	\begin{prop}\label{prop:secondboundaryintegral}
		Suppose that $m/N$ is a square. Let $X \in L_{-m,h}$, let $\ell_{X}$ be the cusp associated to $X$, and let $r_{\ell_{X}}$ be the real part of the geodesic $c(X)$. For $r_{\ell_{X}} < \Re(\sigma_{\ell_{X}}z) < r_{\ell_{X}}+\alpha_{\ell_{X}}$ and $k > 0$ we have
		\begin{align*}
		&\left(2\sqrt{mN}\right)^{-k}\sqrt{N}^{-1}\sum_{\gamma \in \overline{\Gamma}_{\ell_{X}}}\eta(\sigma_{\ell_{X}}^{-1}\gamma X,\tau,z) \\
		&= (-1)^{k}i\frac{1}{\alpha_{\ell_{X}}^{k+1}}\frac{k!}{(2k+1)!}\sum_{d = 0}^{k}\frac{(d+k)!}{(d+1)!}B_{d+1}\left(\frac{x-r_{\ell_{X}}}{\alpha_{\ell_{X}}}\right)\left(-\frac{iy}{\alpha_{\ell_{X}}} \right)^{-k-d-1} \sum_{j=0}^{k-d}\binom{2k+1}{j}\\
		&\quad + (-1)^{k}i\frac{2^{2k-1}}{\alpha_{\ell_{X}}^{k+1}}\frac{k!}{(2k+1)!}\left(\psi^{(k)}\left(\frac{z-r_{\ell_{X}}}{\alpha_{\ell_{X}}}\right)+(-1)^{k}\psi^{(k)}\left(1-\frac{z-r_{\ell_{X}}}{\alpha_{\ell_{X}}}\right) \right)\\
		&\quad -\frac{\pi}{\left(2\sqrt{mN}\right)^{k}\sqrt{N}\alpha_{\ell_{X}}}\left( \frac{i\sqrt{N}}{4\pi \sqrt{m}}\right)^{k+1}v^{-k-1/2}\beta_{3/2+k}^{c}(-4\pi m v)\sum_{\substack{n \in \frac{1}{\alpha_{\ell_{X}}}\Z \\ n < 0}}(4\pi n)^{k}e^{-2\pi i n(z-r_{\ell_{X}})}\\
		&\quad + \sum_{n \in \frac{1}{\alpha_{\ell_{X}}}\Z \setminus \{0\}}g_{n}(y)e^{-2\pi i n (x-r_{\ell_{X}})} + h(y)
		\end{align*}
		where $\beta^c_m(s):=\int_0^1 e^{-st}t^{-m}dt$, $g_{n}(y) = O(e^{-Cy^{2}})$ and $h(y) = O(y^{-k}\log(y))$ as $y\to \infty$, for some constant $C > 0$. 
		
		For $k = 0$ the same formula holds with $h(y)$ replaced by
		\begin{align*}
		\frac{i}{\alpha_{\ell_{X}}}\mathcal{F}(4\pi m v)
		-\frac{i}{\alpha_{\ell_{X}}}\log\left(\frac{y}{\alpha_{\ell_{X}}} \right),
		\end{align*}
		where $\mathcal{F}(w)$ is the function defined in Theorem \ref{theorem fourier expansion shintani}.
	\end{prop}

	
	\begin{proof}
	We abbreviate $\ell = \ell_{X},\alpha = \alpha_{\ell_{X}}$ and $r = r_{\ell_{X}}$, and we write
	\begin{align*}
	\sum_{X \in \overline{\Gamma}_{\ell}}\eta(\sigma_{\ell}^{-1}\gamma X ,\tau,z) = \sum_{n \in \Z}\eta\left(\sqrt{\frac{m}{N}}\begin{pmatrix}-1 & -2(\alpha n -r) \\ 0 & 1\end{pmatrix},\tau,z\right).
	\end{align*}
	Using the formula \eqref{eq eta alternative formula} for $\eta(X,\tau,z)$, we explicitly have
	\begin{align*}
	&\eta\left(\sqrt{\frac{m}{N}}\begin{pmatrix}-1 & -2(\alpha n -r) \\ 0 & 1\end{pmatrix},\tau,z\right) \\
	&= -\frac{2^{k}N^{k+1}}{\big(2\sqrt{mN}(z+\alpha n -r)\big)^{k+1}}\sqrt{v}\int_{1}^{\infty}\int_{2m\frac{(x+\alpha n -r)^{2}}{y^{2}}+2m}^{\infty}t^{k}e^{-2\pi v (t-2m) w^{2}}dt \, dw.
	\end{align*}
	Replacing $z$ by $z+r$ we can assume that $r = 0$. For $s\in\C$ with $\Re(s)>0$ we consider
	\begin{align}
	&-2^{k}N^{k+1}\sqrt{v}\sum_{n \in \Z}\frac{1}{(z+\alpha n)^{k+1}}\int_{0}^{\infty}\int_{2m\frac{(x+\alpha n)^{2}}{y^{2}}+2m}^{\infty}t^{k}e^{-2\pi v (t-2m) w^{2}}dt \, w^{s} \, dw  \label{eq Lambda 1}
\\
	&+2^{k}N^{k+1}\sqrt{v}\sum_{n \in \Z}\frac{1}{(z+\alpha n)^{k+1}}\int_{0}^{1}\int_{2m\frac{(x+\alpha n)^{2}}{y^{2}}+2m}^{\infty}t^{k}e^{-2\pi v (t-2m) w^{2}}dt \, w^{s} \, dw.  \label{eq Lambda 2}
	\end{align}
	Although the difference is holomorphic at $s = 0$, we will see in the computations below that the functions in \eqref{eq Lambda 1} and \eqref{eq Lambda 1} both have a pole of first order at $s = 0$. Thus we compute their constant terms in the Laurent expansion at $s = 0$.
	
	The expression in \eqref{eq Lambda 1} can be evaluated as
	\begin{align*}
		-\frac{2^{k-1}N^{k+1}}{(2\pi v)^{s/2+1/2}}\sqrt{v}\Gamma\left(\frac{s+1}{2}\right) \sum_{n \in \Z}\frac{1}{(z+\alpha n)^{k+1}}\int_{2m\frac{(x+\alpha n)^{2}}{y^{2}}+2m}^{\infty}  t^k (t-2m)^{-s/2-1/2} dt.
	\end{align*}		
	Replacing $t$ by $t+2m$ and expanding $(t+2m)^{k}$ we get after a short calculation
	\begin{align}\label{eq Lambda 3}
		\frac{2^{k}N^{k+1}(2m)^{k-s/2+1/2}}{(2\pi v)^{s/2+1/2}}\sqrt{v}\Gamma\left(\frac{s+1}{2}\right) \sum_{j=0}^{k}\binom{k}{j}\frac{y^{-2j+s-1}}{2j-s+1}\sum_{n \in \Z}\frac{|x+\alpha n|^{2j-s+1}}{(z+\alpha n)^{k+1}}.
	\end{align}

	For $0 < x < \alpha$ we can write
	\begin{align*}
		&\sum_{n\in\Z}\frac{|x+\alpha n|^{2j-s+1}}{(z+\alpha n)^{k+1}} \\
		& =\alpha^{2j-s} \frac{i^k}{k!} \frac{\partial^k}{\partial y^k}\left( \sum_{n\geq 0} \frac{1}{\left(\frac{z}{\alpha}+ n\right)\left(\frac{x}{\alpha} + n\right)^{-2j+s-1}}- \sum_{n\geq 0} \frac{1}{\left(1-\frac{z}{\alpha}+ n\right)\left(1-\frac{x}{\alpha} +n\right)^{-2j+s-1}}\right).
	\end{align*}
	For $m \in \Z$ and $\Re(s) > m$ we have
	\[
	\sum_{n \geq 0}\frac{1}{\left(\frac{z}{\alpha}+n\right)\left(\frac{x}{\alpha}+n \right)^{-m+s-1}}-\sum_{n \geq 0}\frac{1}{\left(\frac{x}{\alpha}+n \right)^{-m+s}} = -\frac{iy}{\alpha}\sum_{n \geq 0}\frac{1}{\left(\frac{z}{\alpha}+n\right)\left(\frac{x}{\alpha}+n \right)^{-m+s}}.
	\]

	Using induction on $m \geq 0$ we find that this equals
\begin{align*}
		\sum_{d=0}^{m}\left(-\frac{iy}{\alpha}\right)^{d} \zeta\left(\frac{x}{\alpha},s-m+d\right) + \left(-\frac{iy}{\alpha}\right)^{m+1}\sum_{n \geq 0}\frac{1}{\left(\frac{z}{\alpha}+n\right)\left(\frac{x}{\alpha}+n \right)^{s}}.
	\end{align*}
	Taking the $k$-th derivative, we obtain 
	\begin{align*}
	&\sum_{n\in\Z}\frac{|x+\alpha n|^{2j-s+1}}{(z+\alpha n)^{k+1}}\\
	 &=\alpha^{2j-s-k} \sum_{d = k}^{2j}\binom{d}{k}\left(-\frac{iy}{\alpha}\right)^{d-k}\left(\zeta\left(\frac{x}{\alpha},s-2j+d \right)+(-1)^{d+1}\zeta\left(1-\frac{x}{\alpha},s-2j+d \right)\right) \\
	&\quad+\alpha^{2j-s-k}\sum_{d = 0}^{\min\{2j+1,k\}}\binom{2j+1}{d}\left(-\frac{iy}{\alpha}\right)^{2j+1-d} \\
	&\qquad \qquad \times \left(\sum_{n \geq 0}\frac{1}{\left(\frac{z}{\alpha}+n\right)^{k+1-d}\left( \frac{x}{\alpha}+n\right)^{s}}+(-1)^{k-d}\sum_{n \geq 0}\frac{1}{\left(1-\frac{z}{\alpha}+n\right)^{k+1-d}\left( 1-\frac{x}{\alpha}+n\right)^{s}}\right),
	\end{align*}
	where the first line on the right-hand side vanishes if $k > 2j$. In the Hurwitz zeta functions we can plug in $s = 0$ and obtain
	\[
	\zeta\left(\frac{x}{\alpha},-2j+d \right)+(-1)^{d+1}\zeta\left(1-\frac{x}{\alpha},-2j+d \right) = - 2\frac{B_{2j-d+1}\left(\frac{x}{\alpha}\right)}{2j-d+1}.
	\]
	The series in the last line converge at $s = 0$ for $0 \leq d \leq k-1$ to the value 
	\begin{align*}
	&\zeta\left(\frac{z}{\alpha},k+1-d\right)+(-1)^{k-d}\zeta\left(1-\frac{z}{\alpha},k+1-d\right) \\
	&= \frac{(-1)^{k+1-d}}{(k-d)!}\left(\psi^{(k-d)}\left(\frac{z}{\alpha}\right) +(-1)^{k-d}\psi^{(k-d)}\left(1-\frac{z}{\alpha}\right)\right).
	\end{align*}
	Further, since $\zeta(z,s)$ has a pole of first order at $s = 1$ with residue $1$ and constant term $-\psi(z)$, we see as in the proof of \cite{bif}, Lemma 8.5, that
	\[
	\sum_{n \geq 0}\frac{1}{\left(\frac{z}{\alpha}+n\right)\left( \frac{x}{\alpha}+n\right)^{s}} = \frac{1}{s} - \psi\left(\frac{z}{\alpha}\right) + O(s).
	\]
	Taking everything together, we obtain that the constant term at $s=0 $ of the expression in \eqref{eq Lambda 3} is given by
	\begin{align*}
	&2^{2k+1}N^{k+1}m^{k+1/2}\alpha^{-k-1}i\sum_{j=\lceil \frac{k}{2} \rceil}^{k}\binom{k}{j}\frac{(-1)^{j}}{2j+1}\sum_{d = k}^{2j}\binom{d}{k}\frac{B_{2j-d+1}\left(\frac{x}{\alpha}\right)}{2j-d+1}\left(-\frac{iy}{\alpha}\right)^{-2j+d-k-1}\\
	&-2^{2k}N^{k+1}m^{k+1/2}\alpha^{-k-1}i\sum_{j=0}^{k}\binom{k}{j}\frac{(-1)^{j}}{2j+1}\sum_{d = 0}^{\min\{2j+1,k\}}\binom{2j+1}{d}\left( -\frac{iy}{\alpha}\right)^{-d} \\
	&\qquad \qquad \qquad \qquad \qquad \qquad \qquad  \times  \frac{(-1)^{k+1-d}}{(k-d)!}\left(\psi^{(k-d)}\left(\frac{z}{\alpha}\right) +(-1)^{k-d}\psi^{(k-d)}\left(1-\frac{z}{\alpha}\right)\right) \\
	&- 2^{2k+1}N^{k+1}m^{k+1/2}\alpha^{-k-1}i\left( -\frac{iy}{\alpha}\right)^{-k}\sum_{j= \lceil \frac{k-1}{2}\rceil}^{k}\binom{k}{j}\binom{2j+1}{k}\frac{(-1)^{j}}{2j+1}\\
	&\qquad \qquad \qquad \qquad \qquad \qquad \qquad \times\left(\log\left( \frac{y}{2\sqrt{\pi m v}\alpha}\right)+\frac{1}{2}\psi\left(\frac{1}{2}\right) + \frac{1}{2j+1} \right).
	\end{align*}
	The last line greatly simplifies for $k = 0$, and for $k > 0$ it decays like $O(y^{-k}\log(y))$ as $y \to \infty$ and contributes to the function $h(y)$. 
	
	The first line can be reordered to
	\begin{align*}
	2^{2k+1}N^{k+1}m^{k+1/2}\alpha^{-k-1}i\sum_{d = 0}^{k}\frac{B_{d+1}\left(\frac{x}{\alpha}\right)}{d+1}\left(-\frac{iy}{\alpha} \right)^{-k-d-1}\sum_{j = \lceil \frac{k}{2}\rceil}^{k}\binom{k}{j}\binom{2j-d}{k}\frac{(-1)^{j}}{2j+1},
	\end{align*}
	and the sum over $j$ can be rewritten using the combinatorial identity
	\begin{align*}
	\sum_{j=\left\lceil \frac{k}{2}\right\rceil}^{k}\binom{k}{j}\binom{2j-d}{k} \frac{(-1)^{j}}{2j+1}= (-1)^{k}\frac{k!(d+k)!}{(2k+1)!d!} \sum_{j=0}^{k-d}\binom{2k+1}{j}
	\end{align*}
	for $0 \leq d \leq k$, which can be proven by induction on $d$ (going from $d = k$ to $d = 0$) and Lemma 14.1 in \cite{borcherds}.
	
	The second line can considerably be simplified. Reordering the summation we get
	\begin{align*}
	&-2^{2k}N^{k+1}m^{k+1/2}\alpha^{-k-1}i\sum_{d = 0}^{k}\left( -\frac{iy}{\alpha}\right)^{-d}\frac{(-1)^{k+1-d}}{(k-d)!}\left(\psi^{(k-d)}\left(\frac{z}{\alpha}\right) +(-1)^{k-d}\psi^{(k-d)}\left(1-\frac{z}{\alpha}\right)\right) \\
	& \qquad \qquad \qquad \qquad \qquad \qquad \qquad \qquad \qquad \qquad  \qquad\times\sum_{j=\lceil \frac{d-1}{2} \rceil}^{k}\binom{k}{j}\binom{2j+1}{d}\frac{(-1)^{j}}{2j+1}.
	\end{align*}
	Using Lemma 14.1 of \cite{borcherds} with $A = 2k-d+1, B = k-d$ and $C = k$ we see that the sum over $j$ vanishes for $1 \leq d \leq k$. For $d = 0$ the sum over $j$ can explicitly be evaluated as
	\[
	\sum_{j = 0}^{k}\binom{k}{j}\frac{(-1)^{j}}{2j+1} = \frac{2^{2k}(k!)^{2}}{(2k+1)!}.
	\]
	This finishes the computation of the constant term at $s = 0$ of the expression in \eqref{eq Lambda 1}.

	
	Using $\int_{C}^{\infty}t^{k}e^{-vt}dt = (-1)^{k}\frac{\partial^{k}}{\partial v^{k}}v^{-1}e^{-Cv}$, the expression in \eqref{eq Lambda 2} can be written as	
	\begin{align}\label{eq Lambda 4}
	\begin{split}
	&-\frac{2^{k}N^{k+1}\sqrt{v}}{(-2\pi)^{k+1}}\int_{0}^{1}e^{4\pi m vw^{2}}w^{s-2k-2} \\
	&\qquad \qquad \qquad \qquad \times \frac{\partial^{k}}{\partial v^{k}}\left(v^{-1}e^{-4 \pi m vw^{2}}\sum_{n \in \Z}\frac{1}{(z+\alpha n)^{k+1}}e^{-4\pi mv \frac{(x+\alpha n)^{2}}{y^{2}}w^{2}}\right) dw.
	\end{split}
	\end{align}

	Next, we would like to apply Poisson summation to the inner sum. The necessary Fourier transform has been computed in \cite{alfesschwagen}, Lemma 5.4. Since the explicit formula and the further computations are quite technical for $k > 0$, we restrict to the case $k = 0$ for simplicity. In this case the inner sum equals
	\begin{align*}
	\sum_{n \in \Z}\frac{1}{z+\alpha n}e^{-4\pi m v \frac{(x+\alpha n)^{2}}{y^{2}}w^{2}} &= \frac{1}{\alpha}\sum_{n \in \frac{1}{\alpha}\Z}e^{-2\pi i nx} \int_{-\infty}^{\infty}\frac{1}{t+iy}e^{-4\pi m v \frac{t^{2}}{y^{2}}w^{2}}e^{2\pi i n t } dt \\
	&= - \frac{i\pi}{\alpha}\sum_{n \in \frac{1}{\alpha}\Z}e^{-2\pi i n z}e^{4\pi m vw^{2}}\erfc\left(2\sqrt{\pi m v}w+ \frac{\pi n y}{2\sqrt{\pi m v}w}\right).
	\end{align*}
	Using $\erfc(-x) = 2-\erfc(x)$ and $\erfc(x) \leq e^{-x^{2}}$ for $x > 0$ we see that \eqref{eq Lambda 4} can be written as
	\begin{align}\label{eq Lambda 5}
	\begin{split}
	&-\frac{i N}{2\alpha\sqrt{v}}\int_{0}^{1}e^{4\pi m v w^{2}}\erfc\left(2\sqrt{\pi m v}w\right)w^{s-2}dw \\
	&-\frac{i N}{\alpha\sqrt{v}}\sum_{\substack{n \in \frac{1}{\alpha}\Z \\ n < 0}}e^{-2\pi i n z}\int_{0}^{1}e^{4\pi m v w^{2}}w^{s-2}dw + \sum_{n \in \frac{1}{\alpha}\Z\setminus \{0\}}\tilde{g}_{n}(y,s)e^{-2\pi i n x} ,
	\end{split}
	\end{align}
	where $\tilde{g}_{n}(y,s)$ is holomorphic at $s = 0$ and satisfies $\tilde{g}_{n}(y,0) = O(e^{-Cy^{2}})$ as $y \to \infty$, for some constant $C > 0$. The integral in the second line has a holomorphic continuation to $s = 0$ which equals the function $\frac{1}{2}\beta_{3/2}^{c}(-4\pi m v)$.

	Further, we use integration by parts to compute
	\begin{align*}
	&\int_{0}^{1}e^{4\pi m vw^{2}}\erfc\left(2\sqrt{\pi m v}w\right)w^{s-2}dw \\
	& \quad = \frac{1}{s-1}e^{4\pi m v}\erfc\left(2\sqrt{\pi m v}\right) - \frac{8\pi m v}{s-1}\int_{0}^{1}e^{-4\pi m vw^{2}}\erfc\left(2\sqrt{\pi m v}w\right)w^{s}dw+ \frac{4\sqrt{ m v}}{(s-1)s}.
	\end{align*}
	Taking the constant term at $s = 0$ in \eqref{eq Lambda 5} we arrive at
	\begin{align*}
	&-\frac{i N}{2\alpha \sqrt{v}}\left(-e^{4\pi m v}\erfc\left(2\sqrt{\pi m v}\right) + 8\pi m v \int_{0}^{1}e^{-4\pi m v w^{2}}\erfc\left(2\sqrt{\pi m v} w \right)dw - 4\sqrt{mv}\right) \\
	&-\frac{i N}{2\alpha \sqrt{v}}\sum_{\substack{n \in \frac{1}{\alpha}\Z \\ n < 0}}e^{-2\pi i n z}\beta_{3/2}^{c}(4\pi m v) + \sum_{n \in \frac{1}{\alpha}\Z\setminus \{0\}}\tilde{g}_{n}(y,0)e^{-2\pi i n x}.
	\end{align*}
	Collecting everything together, we finally obtain the stated formula.

	\end{proof}

	Using Propositions \ref{prop:firstboundaryintegral} and \ref{prop:secondboundaryintegral} we can now evaluate the integrals in \eqref{eq unwinded boundary integral 1} and \eqref{eq unwinded boundary integral 2}. We obtain the complementary trace function and a contribution to the regularized cycle integral in the form given in Proposition \ref{prop regularized cycle integrals alternative}. Further, the following lemma shows that the correction terms in the second line of \eqref{eq coefficients of positive index} cancel out nicely with suitable parts of the integrals in \eqref{eq unwinded boundary integral 1} and \eqref{eq unwinded boundary integral 2}.
	

	\begin{lem}[\cite{bif}, Lemma 8.2]
		For $m > 0$ we have
		\begin{align*}
		&\frac{(-\sqrt{2N}i)^{k}}{2\sqrt{m}(8\pi v)^{k/2}}H_{k}\left(2\sqrt{\pi mv}\right)\sum_{X \in \Gamma \backslash L_{-m,h}}\left(a^{\pm}_{\ell_{X}}(0) + (-1)^{k}a^{\pm}_{\ell_{-X}}(0)\right)\\
		&= \frac{1}{\sqrt{N}}\sum_{\ell \in \Gamma \backslash \Iso(V)}\varepsilon_{\ell}a^{\pm}_{\ell}(0)b_{\ell,k}(m,h).
		\end{align*}
	\end{lem}
	
	This finishes the computation of the Fourier coefficients of $\ISh(G,\tau)$ of positive index.


\bibliographystyle{alpha}
\bibliography{bib.bib}

\end{document}